\documentclass[A4,12pt]{article}
\usepackage{amssymb}
\usepackage{amsmath}
\usepackage{amsthm}
\usepackage[nooneline]{subfigure}
\usepackage{graphicx}
\usepackage{varwidth}
\usepackage{float}
\usepackage{fullpage} 
\usepackage{color}

\newtheorem{thm}{Theorem}[section]
\newtheorem{prop}[thm]{Proposition}
\newtheorem{lemma}[thm]{Lemma}

\newtheorem{rem}[thm]{Remark}
\newtheorem{defi}[thm]{Definition}

\theoremstyle{definition}
\newtheorem{example}{Example}

\providecommand{\argmin}{\operatorname{arg\,min}\enskip}
\newcommand\M{\mathcal{M}}
\newcommand{\R}{\mathbb{R}}

\newcommand{\Rext}{\mathbb{R} \cup \{ \infty \} }
\newcommand{\rslv}[1]{\left(I + #1\right)^{-1}}
\newcommand{\norm}[1]{\Vert #1 \Vert}

\providecommand{\iprod}[2]{\langle#1,#2\rangle}

\newcommand\figDama
{
  \begin{figure}[t!]
    \begin{center}
      \setlength{\tabcolsep}{.7mm}
      \begin{tabular}{ccc}
        \includegraphics[width=0.32\textwidth]{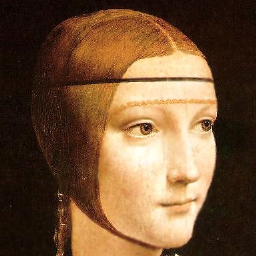}&
        \includegraphics[width=0.32\textwidth]{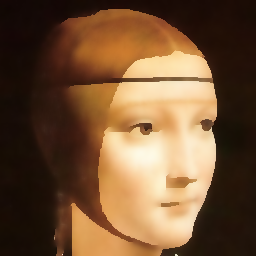}&
        \includegraphics[width=0.32\textwidth]{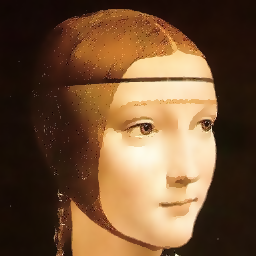}\\

        Input Image. & $u^0=0.$ & $u^0=f.$
      \end{tabular}
    \end{center}
    \vspace{-3mm}
    \caption{
      We compute a piecewise smooth approximation of the input image by minimizing the functional \eqref{eq:MSDenoising}. The parameters were $\alpha = 10$, $\lambda = 0.1$ and $\varepsilon_0 = 0.5$. Unlike the approach \cite{MS-tr} we use completely constant step sizes $\tau$ and $\sigma$. We see that the end result is dependent on the initialization. Using a smooth initialization $u^0=0$ yields an overall smoother result than setting $u^0=f$.  
    }
    \label{fig:dama}
  \end{figure}
}

\newcommand\figTruncQuad
{
  \begin{figure}[t!]
    \begin{center}
      \setlength{\tabcolsep}{.7mm}
      \begin{tabular}{c}
        \includegraphics[width=0.8\textwidth]{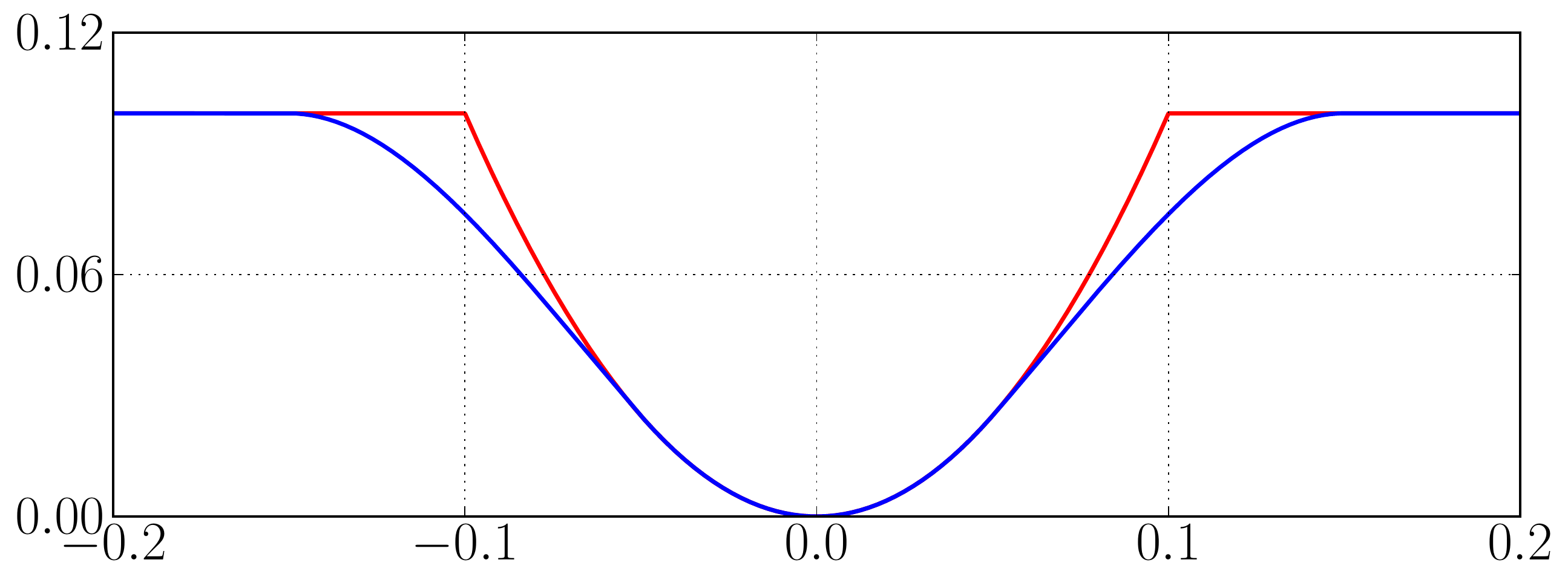}
      \end{tabular}
    \end{center}
    \vspace{-3mm}
    \caption{
      We show the truncated quadratic regularizer \eqref{eq:truncQuadSmooth} along with the smoothed $\omega$-semiconvex variant given in \cite{fornasier}. Parameters are $\alpha = 10$, $\lambda=0.1$ and $\varepsilon_0 = 0.5$.
    }
    \label{fig:truncQuad}
  \end{figure}
}

\newcommand\figCracktip
{
  \begin{figure}[t!]
    \begin{center}
      \setlength{\tabcolsep}{.7mm}
      \begin{tabular}{ccccc}
        \includegraphics[width=0.19\textwidth]{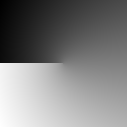}&
        \includegraphics[width=0.19\textwidth]{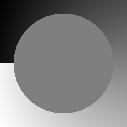}&
        \includegraphics[width=0.19\textwidth]{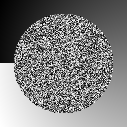}&
        \includegraphics[width=0.19\textwidth]{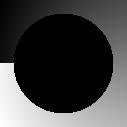}&
        \includegraphics[width=0.19\textwidth]{init_zero.png}\\

        Global Minimum & Initialization & Initialization & Initialization & Initialization \\\\

        \includegraphics[width=0.19\textwidth]{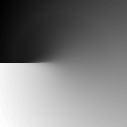}&
        \includegraphics[width=0.19\textwidth]{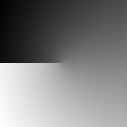}&
        \includegraphics[width=0.19\textwidth]{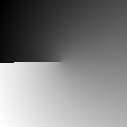}&
        \includegraphics[width=0.19\textwidth]{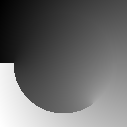}&
        \includegraphics[width=0.19\textwidth]{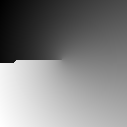}\\

        Relaxation \cite{Pock-et-al-iccv09} & $\sigma_0=2$, $\gamma=\frac{3}{40}$ & $\sigma_0=2$, $\gamma=\frac{3}{40}$ & $\sigma_0=2$, $\gamma=\frac{3}{40}$ & $\sigma_0=1$, $\gamma=\frac{1}{8}$
      \end{tabular}
    \end{center}
    \vspace{-3mm}
    \caption{
      In this figure we show the Mumford-Shah inpainting results obtained with the nonconvex PDHG algorithm. The first row shows the analytical solution to the cracktip problem along with the different initializations. The second row shows the results obtained by the nonconvex primal-dual algorithm next to the convex relaxation \cite{Pock-et-al-iccv09}. Note that the results obtained by the nonconvex primal-dual algorithm heavily depend on the initialization and chosen step size scheme.
    }
    \label{fig:cracktip}
  \end{figure}
}

\newcommand\figDithering
{
  \begin{figure}[t!]
    \begin{center}
      \setlength{\tabcolsep}{.7mm}
      \begin{tabular}{cccc}
        \includegraphics[width=0.24\textwidth]{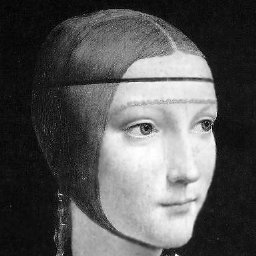}&
        \includegraphics[width=0.24\textwidth]{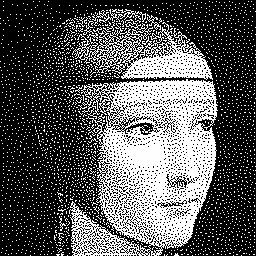}&
        \includegraphics[width=0.24\textwidth]{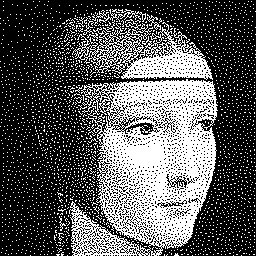}&
        \includegraphics[width=0.24\textwidth]{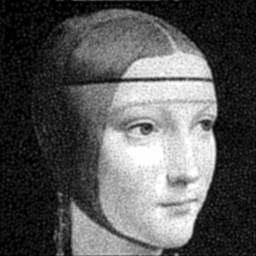}\\
        Input Image $u^0$. & Result $\widehat u$ on \eqref{eq:dithering}. & Binarized $\widehat u$. & $k * \widehat u$. 
      \end{tabular}
    \end{center}
    \vspace{-3mm}
    \caption{Variational Image Dithering. \textbf{From left to right:} We show the input image along with the result of minimizing the nonconvex cost functional \eqref{eq:dithering} using the primal-dual algorithm and a thresholded version of the minimizer. Notice that the result is almost binary. On the right we show the minimizer convolved with the gaussian kernel $k$.}
    \label{fig:dithering}
  \end{figure}
}

\newcommand\figConvergences
{
  \begin{figure}[h!]
    \begin{center}
      \setlength{\tabcolsep}{.7mm}
      \begin{tabular}{cc}
        \includegraphics[width=0.48\textwidth]{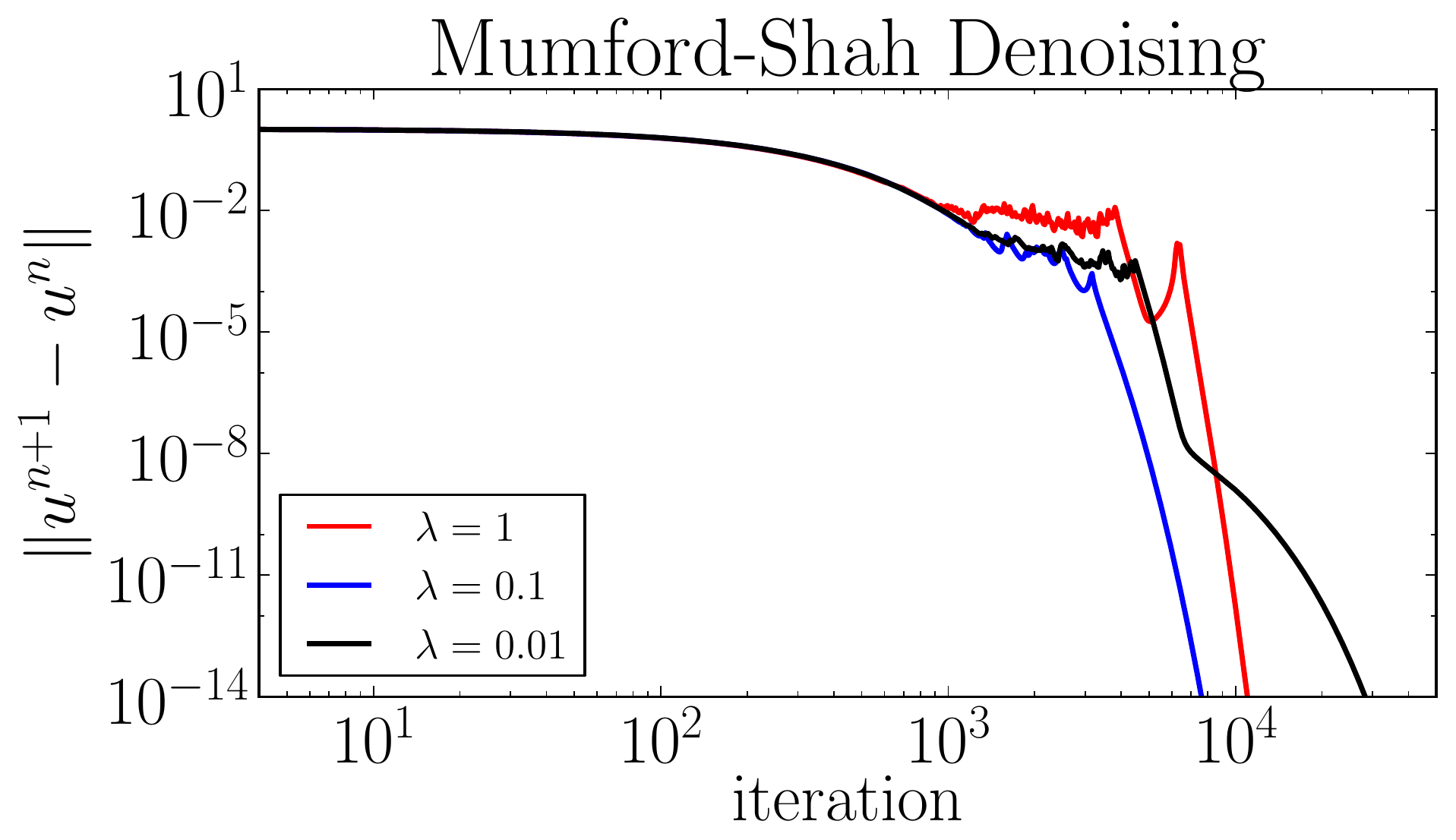}&
        \includegraphics[width=0.48\textwidth]{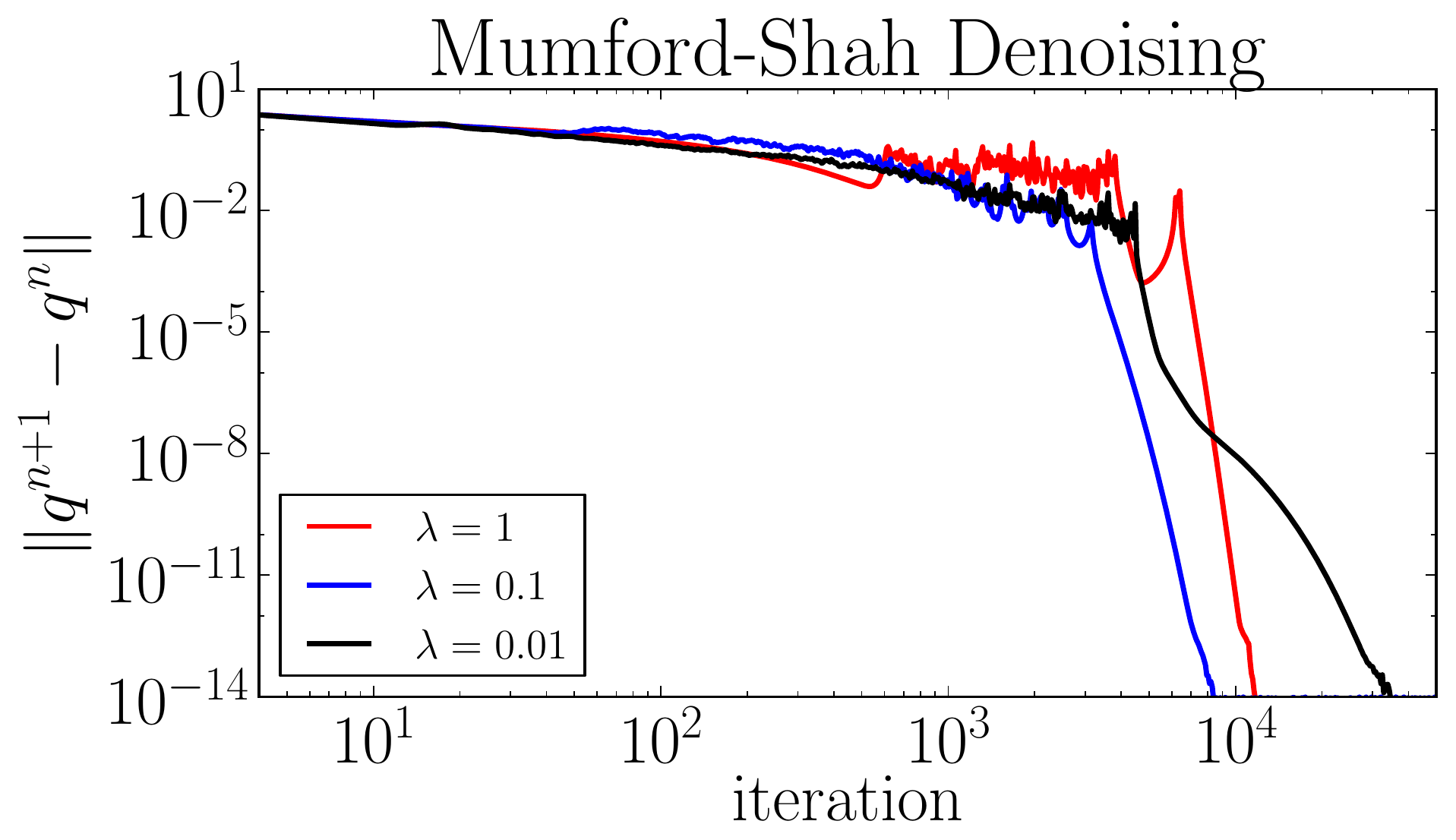}\\

        \includegraphics[width=0.48\textwidth]{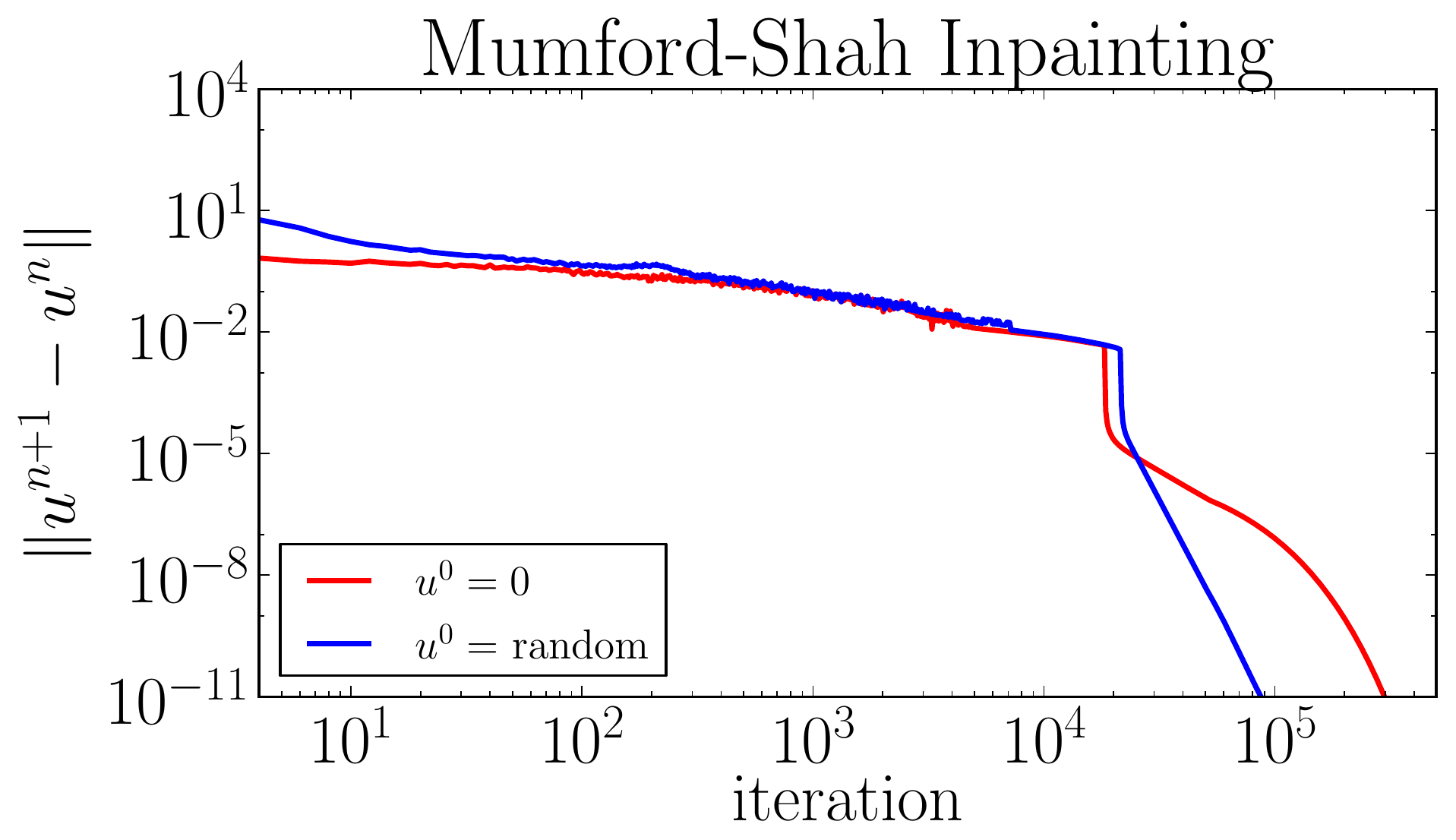}&
        \includegraphics[width=0.48\textwidth]{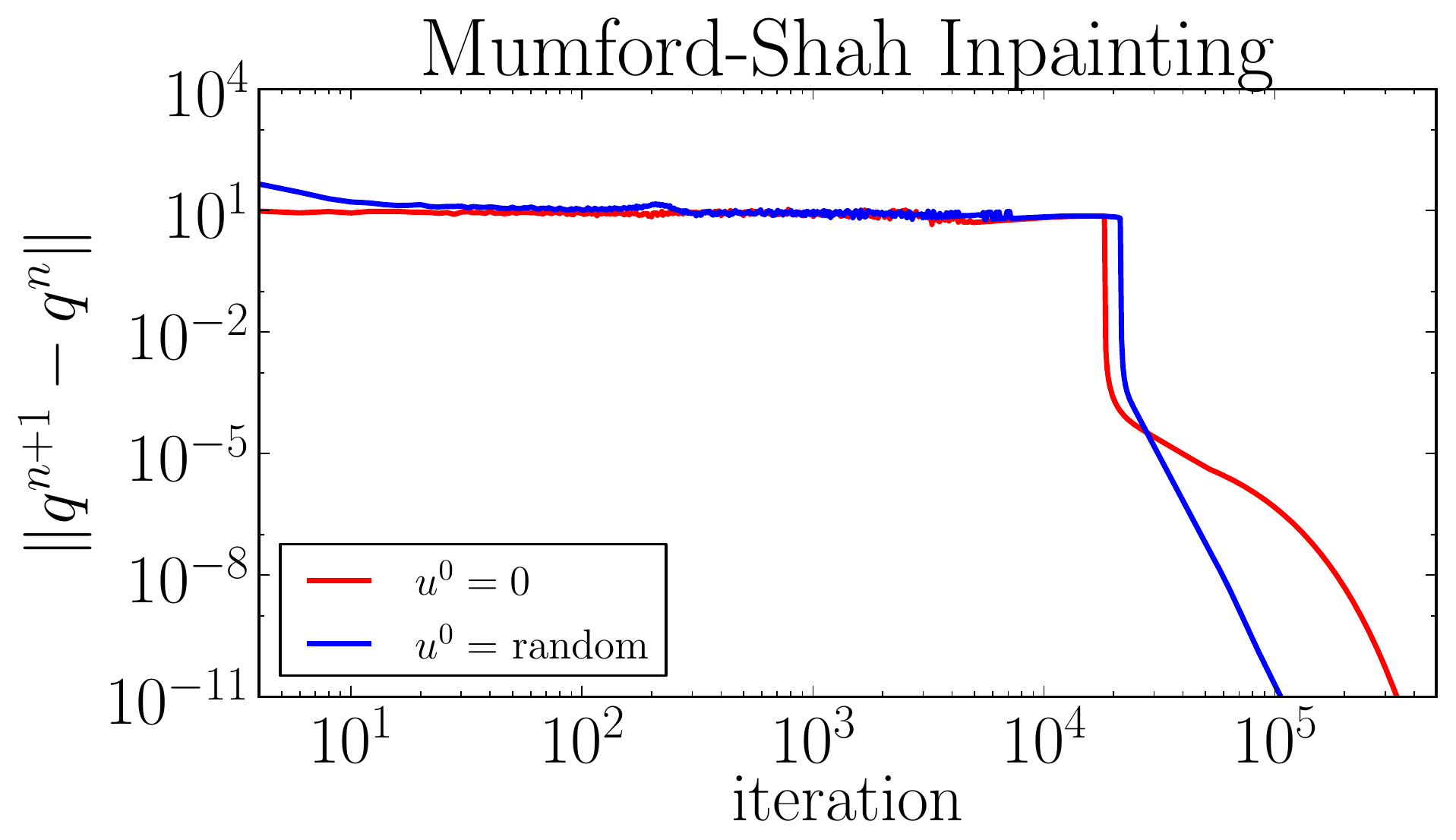}\\

        \includegraphics[width=0.48\textwidth]{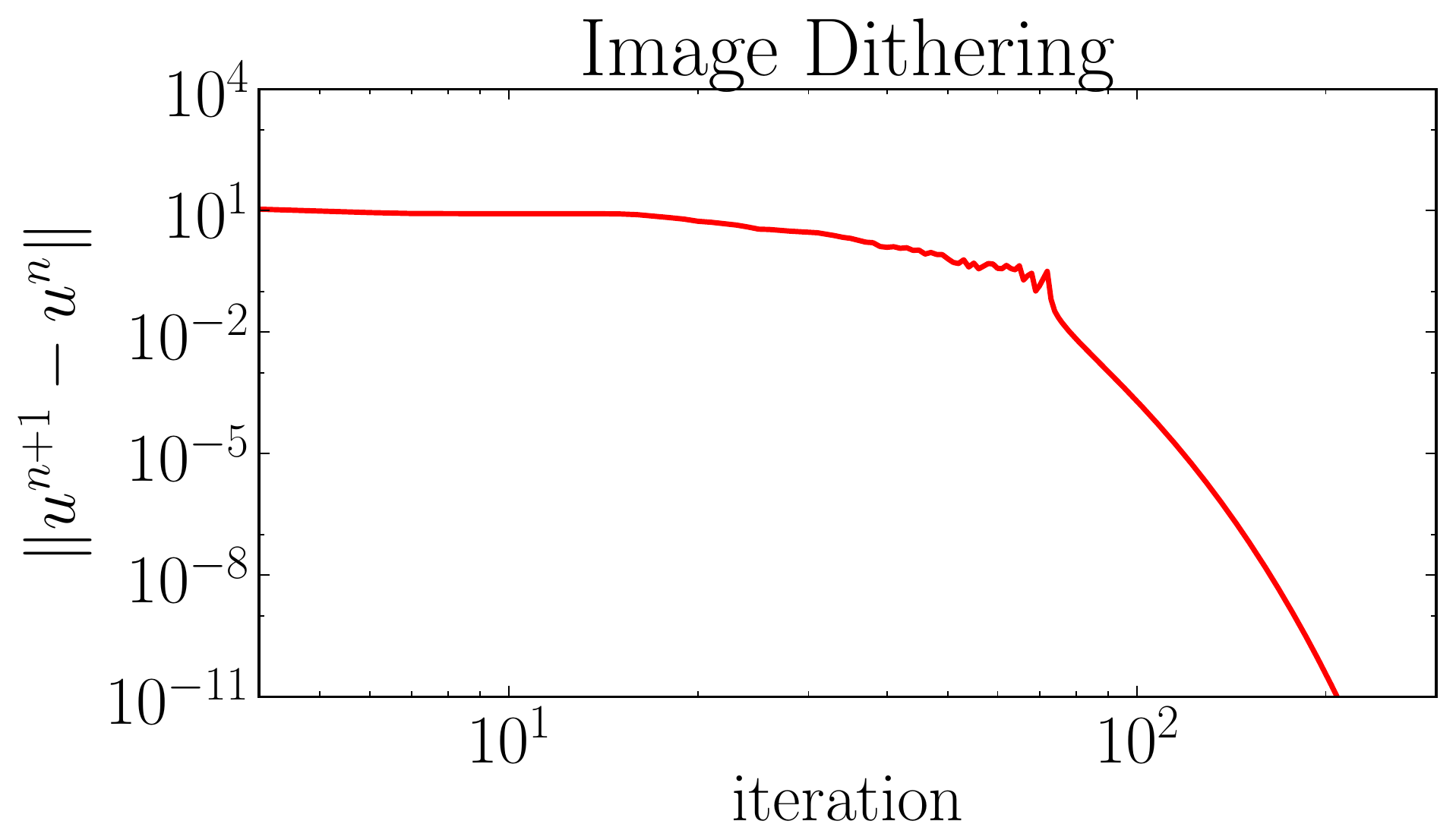}&
        \includegraphics[width=0.48\textwidth]{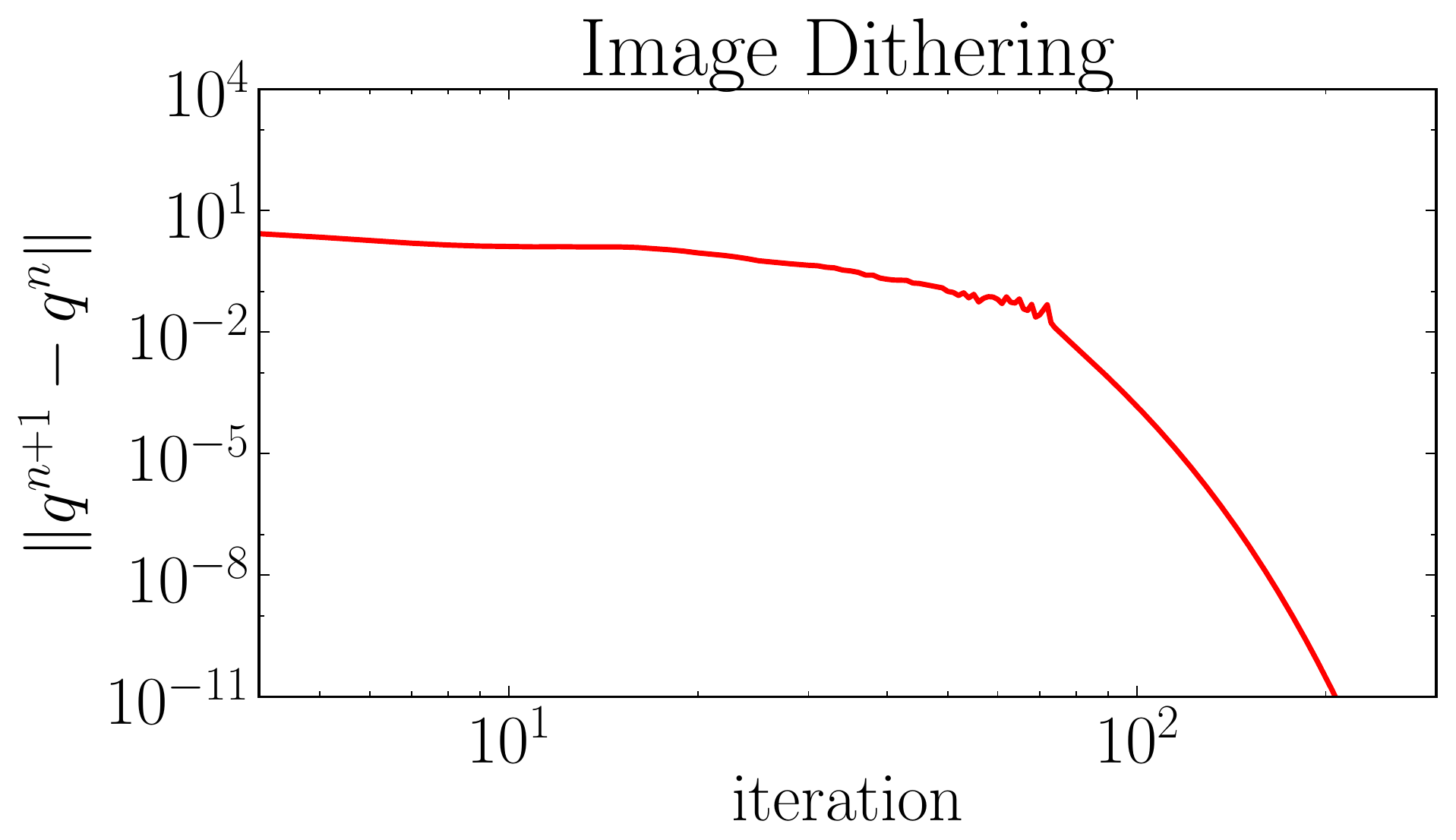}\\
      \end{tabular}
    \end{center}
    \vspace{-3mm}

    \caption{
      We experimentally show that the distance between the iterates $u^n$ and $q^n$ approaches zero. Using Proposition~\ref{lem:distances} we can indeed verify that the algorithm converges to a critical point of the according energy functional in that specific example. This suggests the conjecture that the algorithm also converges energies which are overall $\omega$-semiconvex.
    }

    \label{fig:Convergences}
  \end{figure}
}

\newcommand\figSubdifferential
{
\begin{figure}[t!]
  \begin{center}
    \subfigure[The subdifferential of a function which is differentiable in the classical sense, consists of only one element which is the derivative of the function.]{\includegraphics[width=0.31\textwidth, natwidth=560,natheight=420]{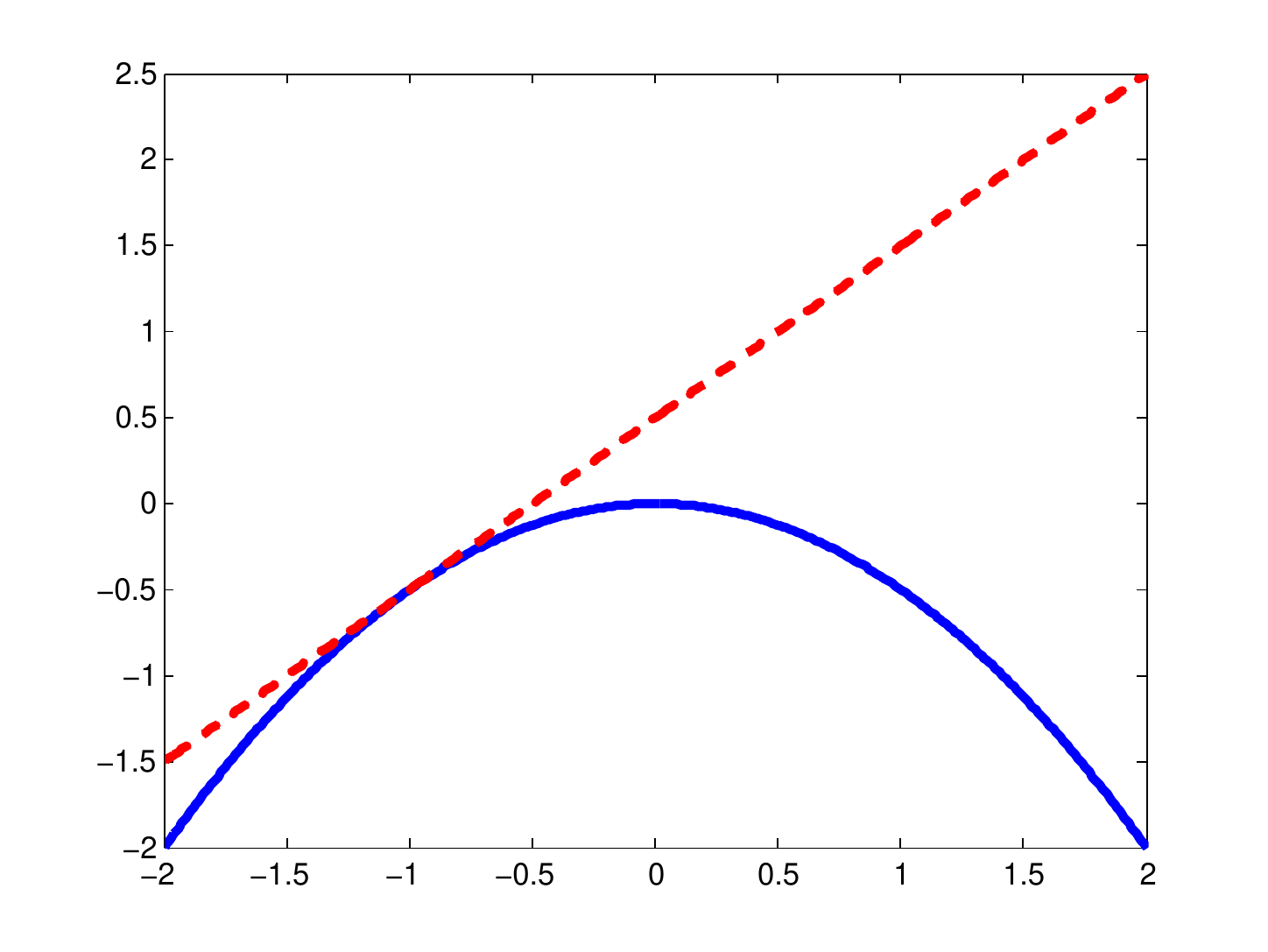}} \hspace{0.25cm}
    \subfigure[{For convex functions, definition \ref{def:subdiff} coincides with the usual definition of the subdifferential of convex functions. For the absolute value function illustrated above, we have $\partial |0| = [-1,1]$}]{\includegraphics[width=0.31\textwidth, natwidth=560,natheight=420]{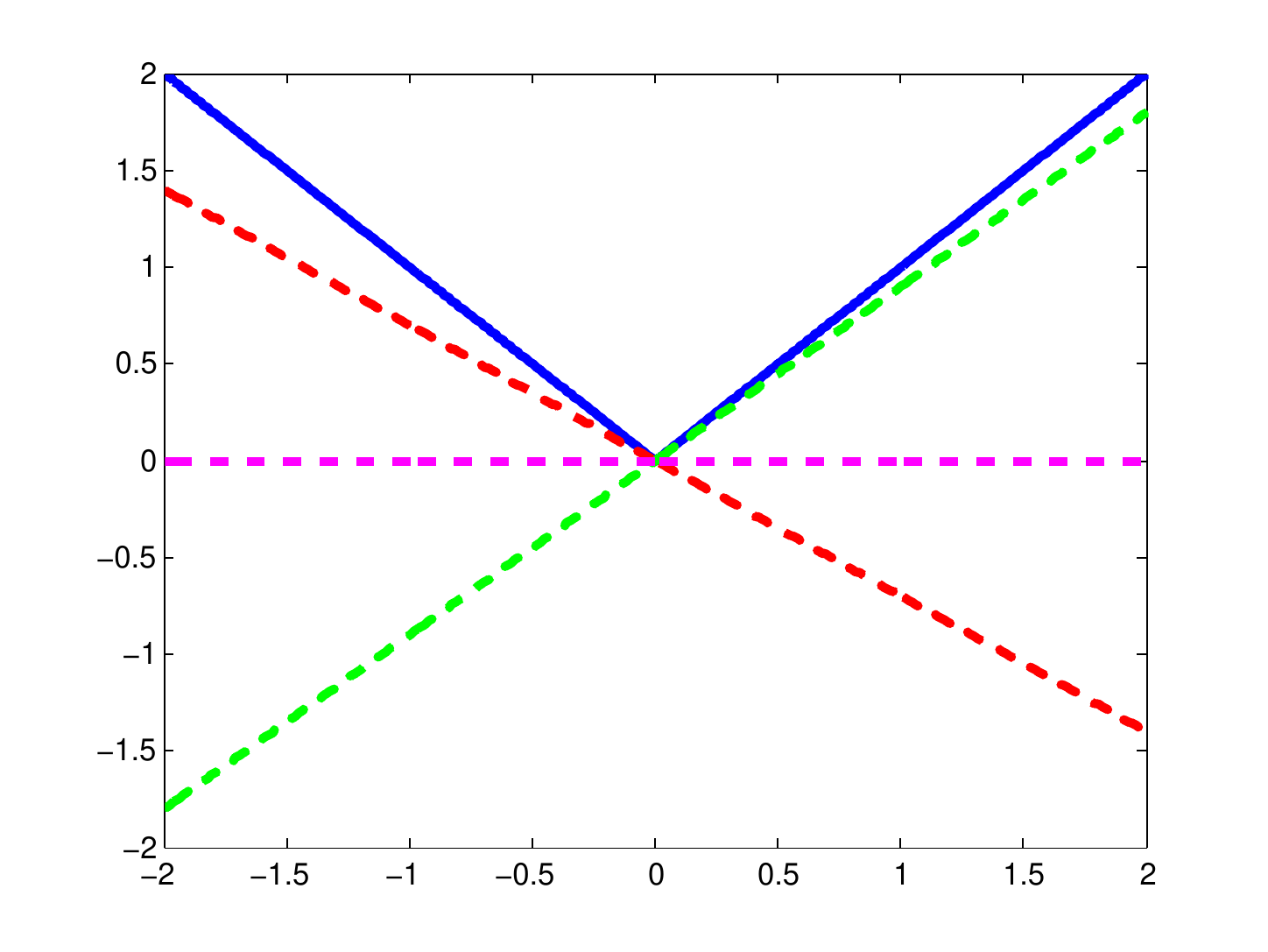}} \hspace{0.25cm}
    \subfigure[{Illustration of the subgradients of a nonconvex function. In our example, $ \partial E(0) = (-\infty, 1] $ due to the lower semicontinuoity at the discontinuity.}]{\includegraphics[width=0.31\textwidth, natwidth=560,natheight=420]{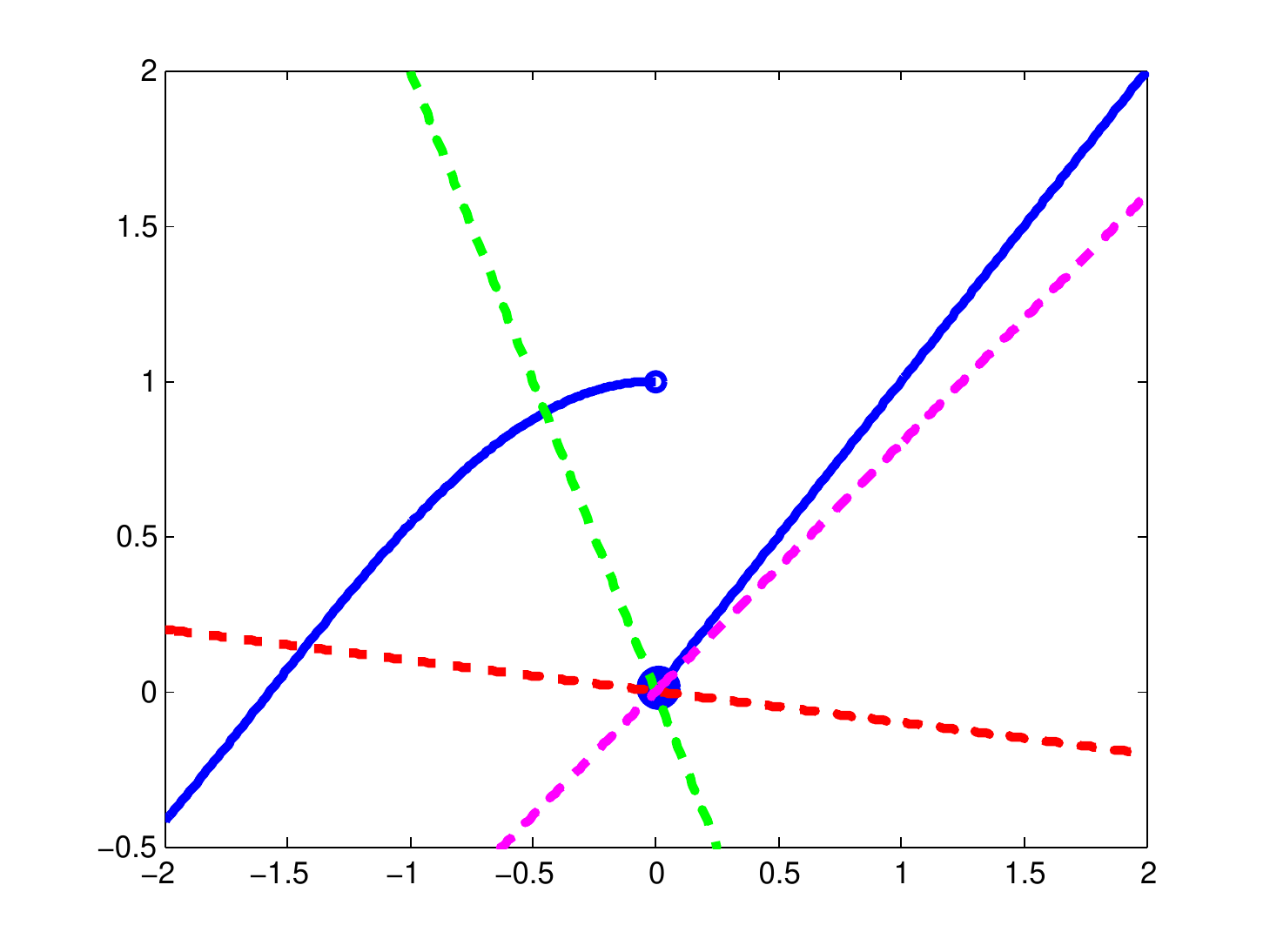}}
    \caption{Examples of subgradients of one dimensional lower semicontinuous functions. The subgradients are illustrated as tangent lines where the slope is an element of the subdifferential.}
    \label{fig:1dSubgrads}
  \end{center}
\end{figure}
}

\newcommand\figBoundedness
{
\begin{figure}[t!]
  \begin{center}
    \includegraphics[width=0.6\textwidth]{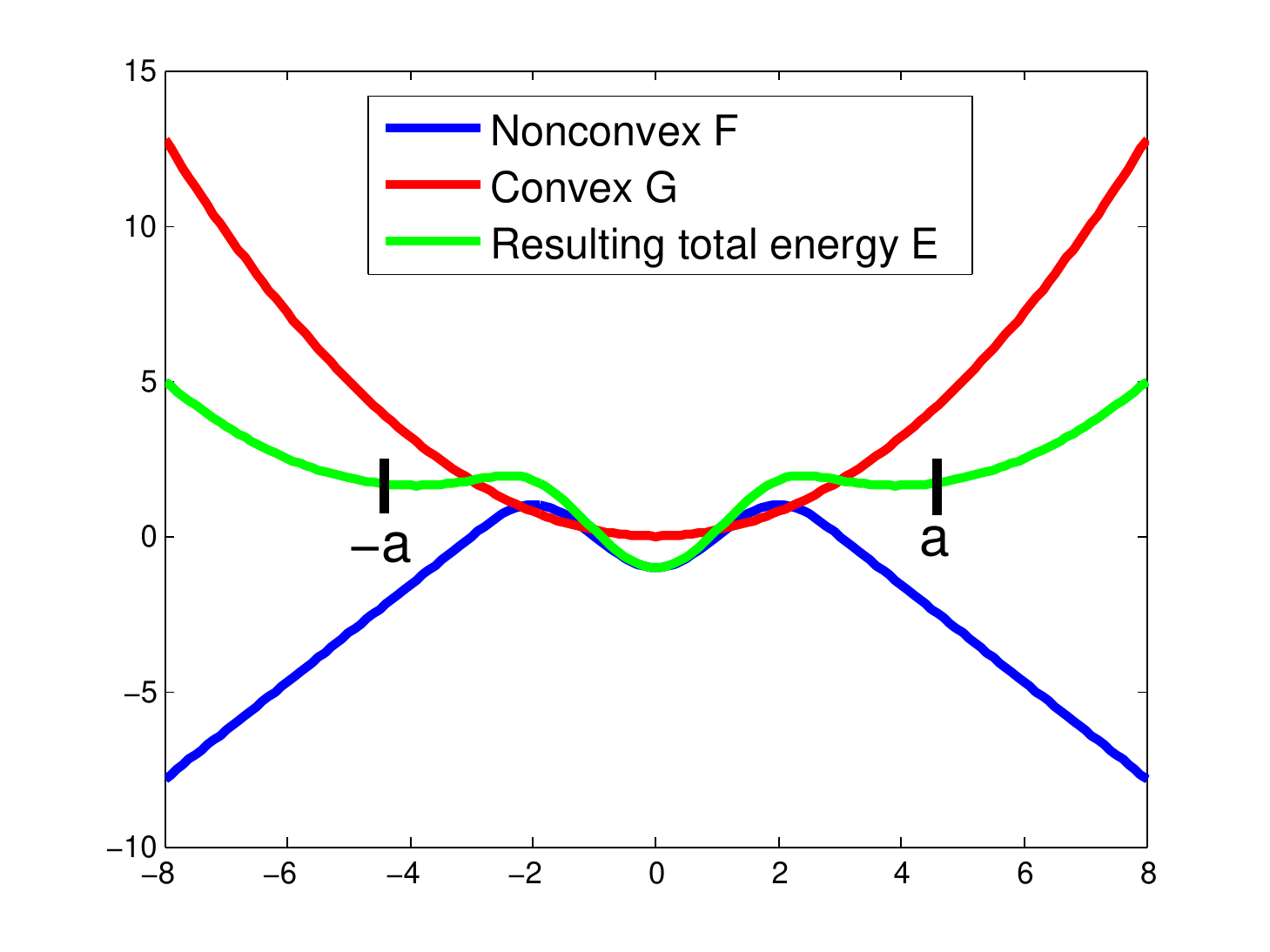}
    \caption{Illustration of the conditions of Proposition \ref{prop:boundedness} in a simple 1D case. We have a nonconvex part (blue) and a convex part (red) whose sum results in a total energy (green) which is nonconvex. In this case it is important that at some point $a$ the worst case slope of the nonconvex function (attained here in the linear parts) is dominated by the convex function. We can see that this is a sufficient condition for the total energy being coercive. }
    \label{fig:boundednessAssumption}
  \end{center}
\end{figure}
}

\newcommand\figDenoiseAndSharpen
{
\begin{figure}[H]
  \begin{center}
    \subfigure[{TV flow 0 iterations}]{\includegraphics[width=0.32\textwidth, natwidth=560,natheight=420]{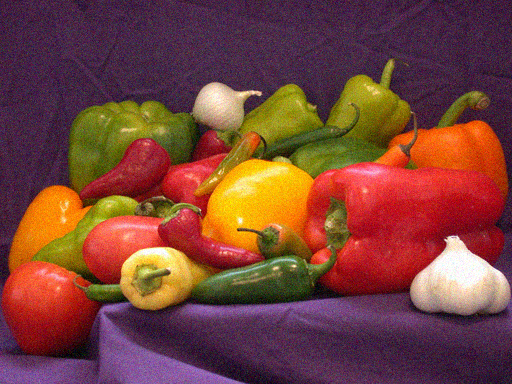}} 
    \subfigure[{TV flow 1 iteration}]{\includegraphics[width=0.32\textwidth, natwidth=560,natheight=420]{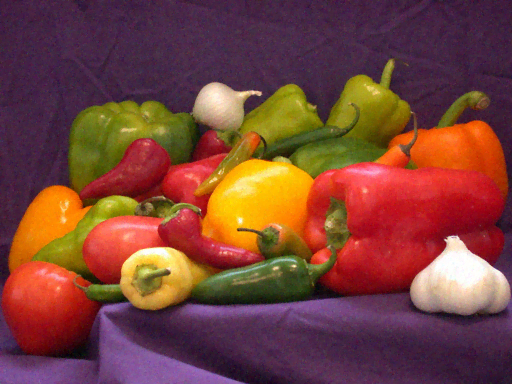}} 
        \subfigure[{TV flow 2 iterations}]{\includegraphics[width=0.32\textwidth, natwidth=560,natheight=420]{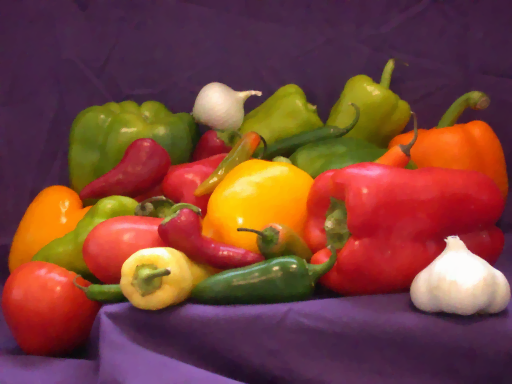}} \\   
        \subfigure[{Enhanced TV flow 0 iterations}]{\includegraphics[width=0.32\textwidth, natwidth=560,natheight=420]{peppers_noisy.png}}
            \subfigure[{Enhanced TV flow 1 iteration}]{\includegraphics[width=0.32\textwidth, natwidth=560,natheight=420]{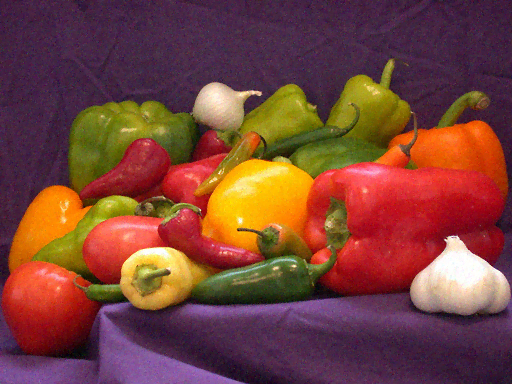}}
    \subfigure[{Enhanced TV flow 2 iterations}]{\includegraphics[width=0.32\textwidth, natwidth=560,natheight=420]{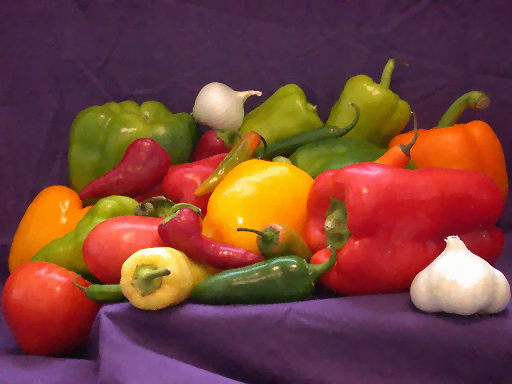}} \\
        \subfigure[{TV flow 3 iterations}]{\includegraphics[width=0.32\textwidth, natwidth=560,natheight=420]{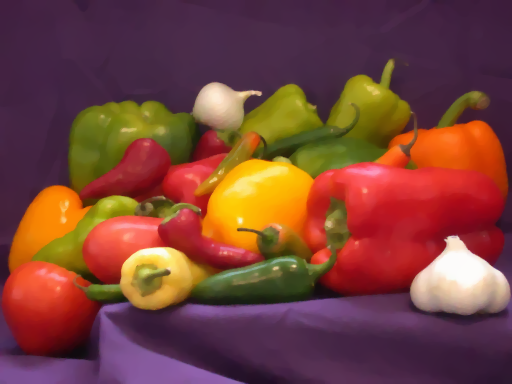}} 
    \subfigure[{TV flow 4 iterations}]{\includegraphics[width=0.32\textwidth, natwidth=560,natheight=420]{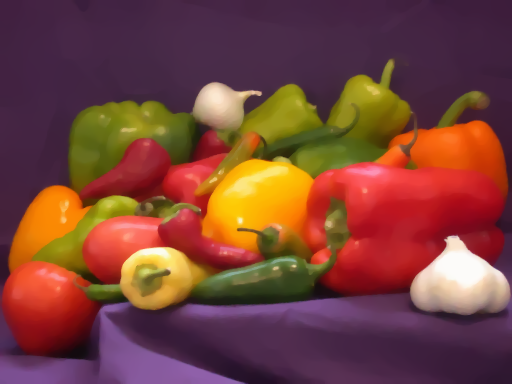}} 
        \subfigure[{TV flow 5 iterations}]{\includegraphics[width=0.32\textwidth, natwidth=560,natheight=420]{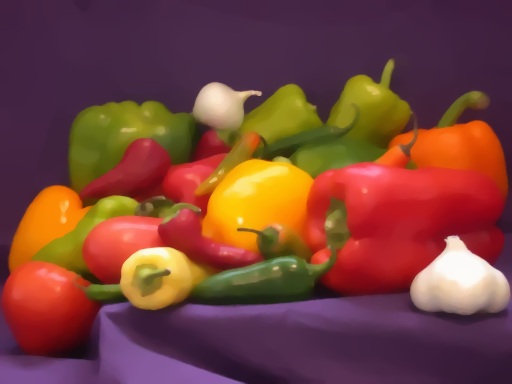}} \\   
        \subfigure[{Enhanced TV flow 3 iterations}]{\includegraphics[width=0.32\textwidth, natwidth=560,natheight=420]{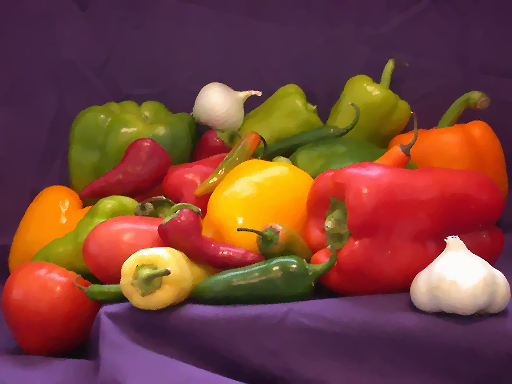}}
            \subfigure[{Enhanced TV flow 4 iterations}]{\includegraphics[width=0.32\textwidth, natwidth=560,natheight=420]{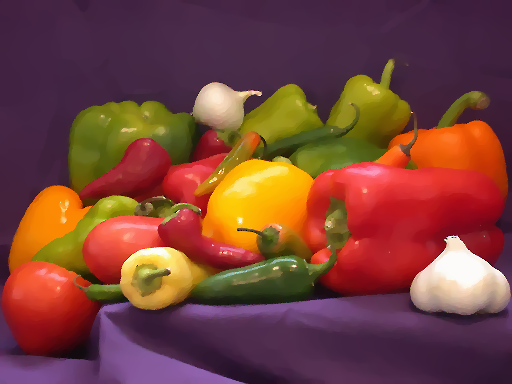}}
    \subfigure[{Enhanced TV flow 5 iterations}]{\includegraphics[width=0.32\textwidth, natwidth=560,natheight=420]{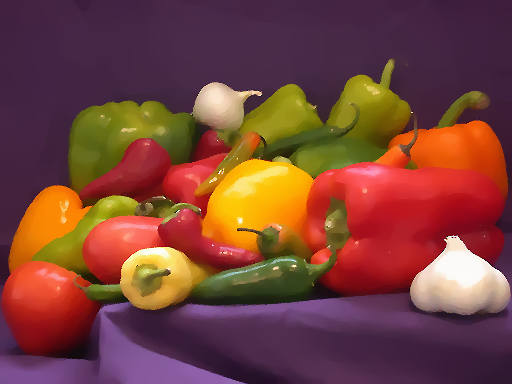}}    
    \caption{Example of incorporating a nonconvex enhancement term into the TV-flow.} 
    
    \label{fig:denoiseAndSharpen}
  \end{center}
\end{figure}
}

\newcommand\figShapeOfN
{
\begin{figure}[H]
  \begin{center}
  {\includegraphics[width=0.45\textwidth, natwidth=560,natheight=420]{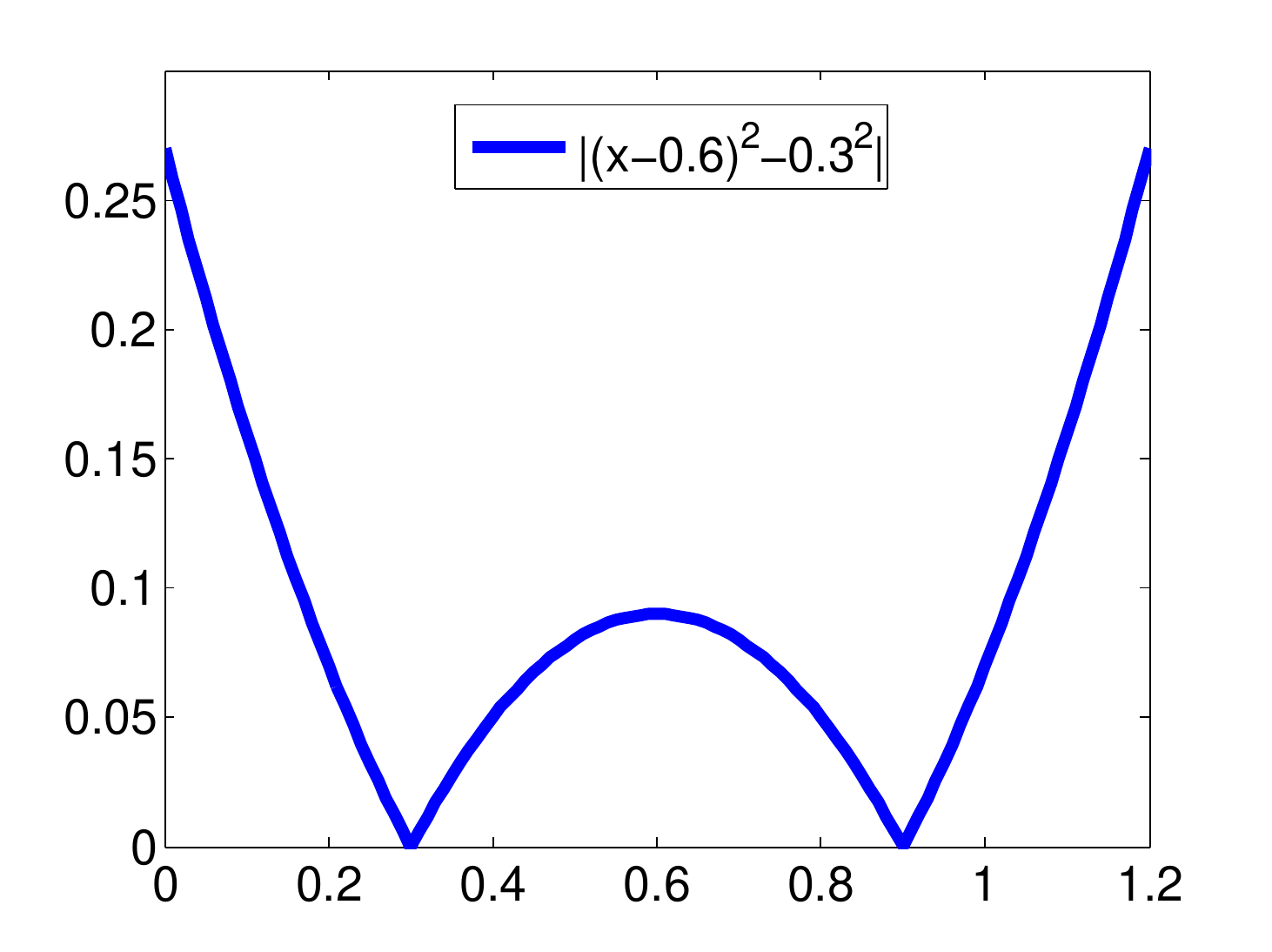}}
  \caption{One dimensional example for the function $N(Au)$ discussed above. $N$ is nonconvex, nonsmooth, but $\omega$-semiconvex.}
    \label{fig:n_of_u_example}
  \end{center}
\end{figure}
}

\newcommand\figIllucorrectionOne
{
\begin{figure}[b!]
    \subfigure[{Original noisy image.}]{\includegraphics[width=0.32\textwidth]{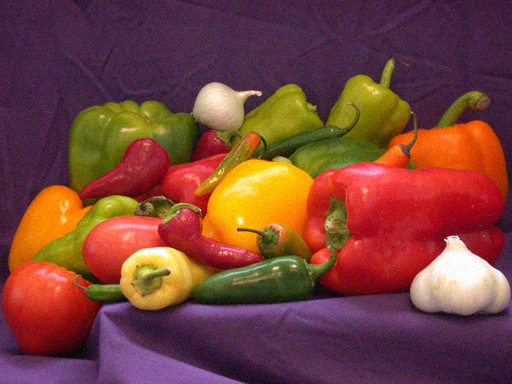}}
    \subfigure[{TV denoised image.}]{\includegraphics[width=0.32\textwidth, natwidth=560,natheight=420]{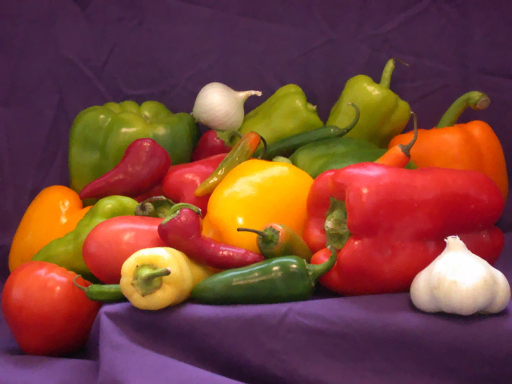}}
    \subfigure[Image after TV denoising along with encouraging intensity values to be either 0.3 or 0.9.]{\includegraphics[width=0.32\textwidth, natwidth=560,natheight=420]{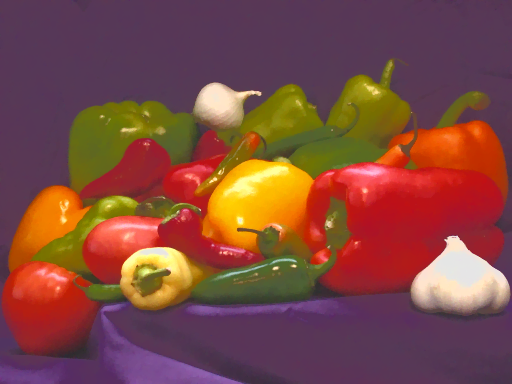}} \\
    \subfigure[Sorted intensity distribution of the pixels of the original image.]{\includegraphics[width=0.32\textwidth, natwidth=560,natheight=420]{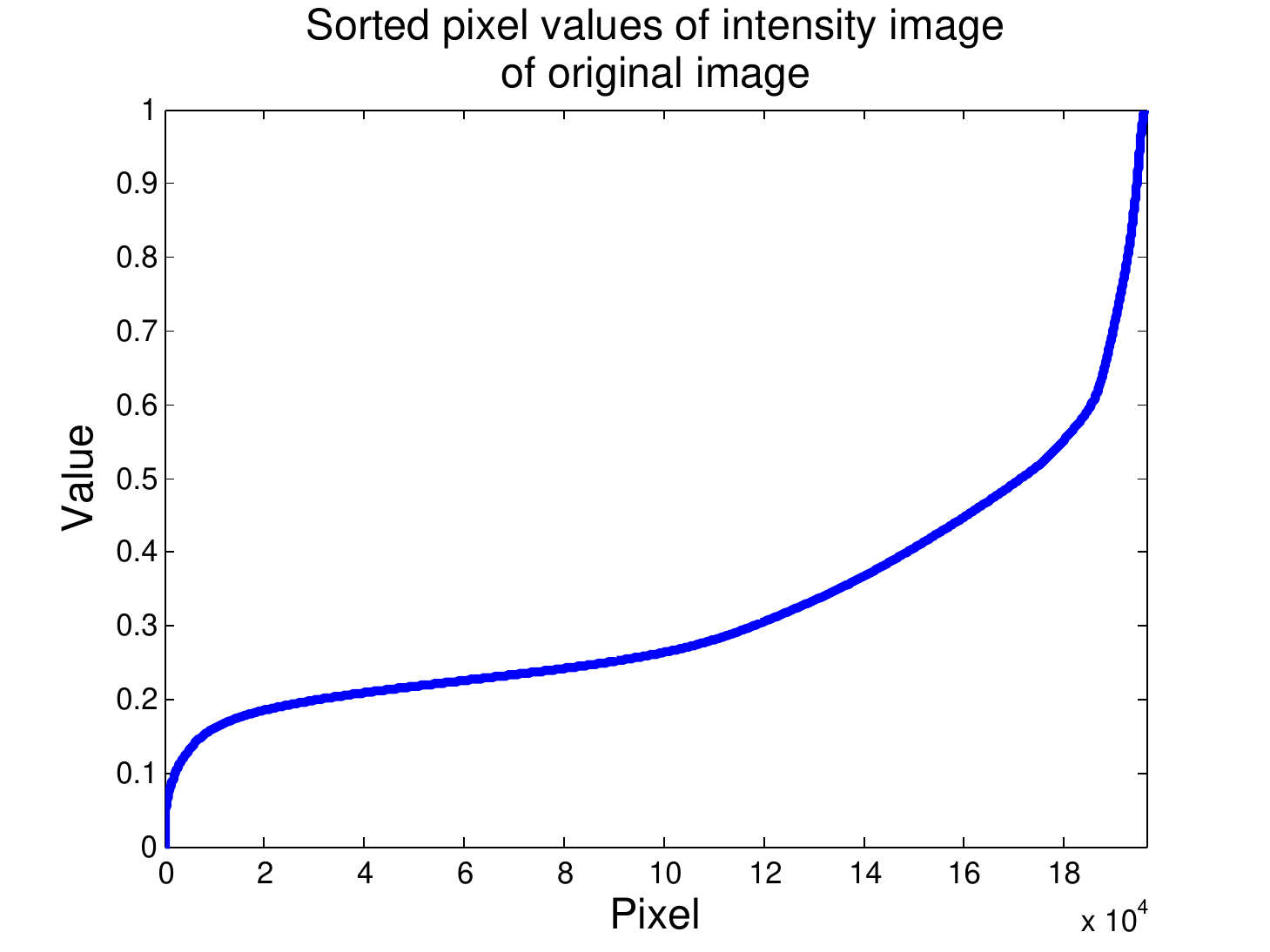}} 
    \subfigure[Sorted intensity distribution of the pixels of the TV denoised image.]{\includegraphics[width=0.32\textwidth, natwidth=560,natheight=420]{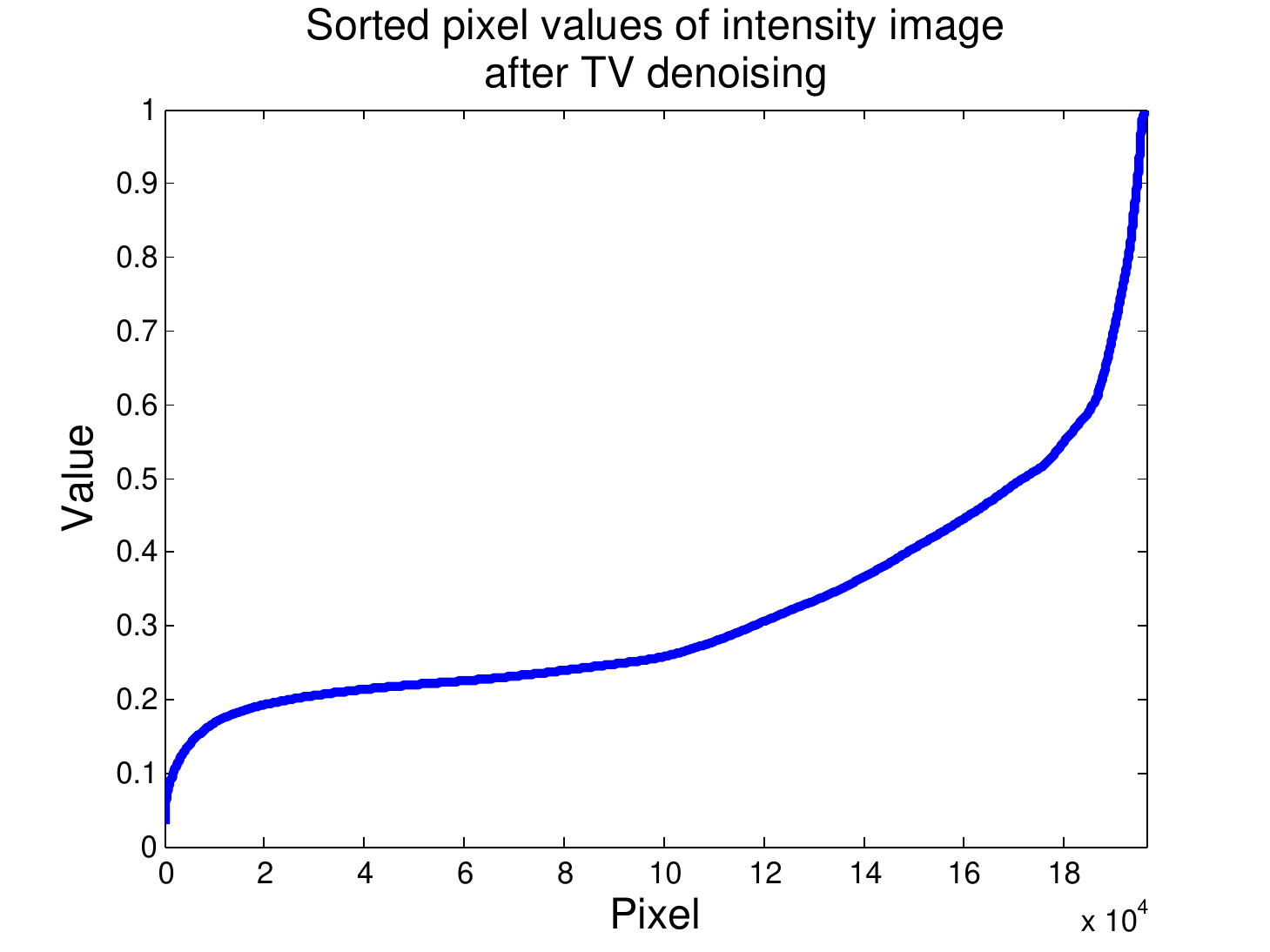}} 
    \subfigure[Sorted intensity distribution of the pixels of the image reconstructed with the nonconvex term.]{\includegraphics[width=0.32\textwidth, natwidth=560,natheight=420]{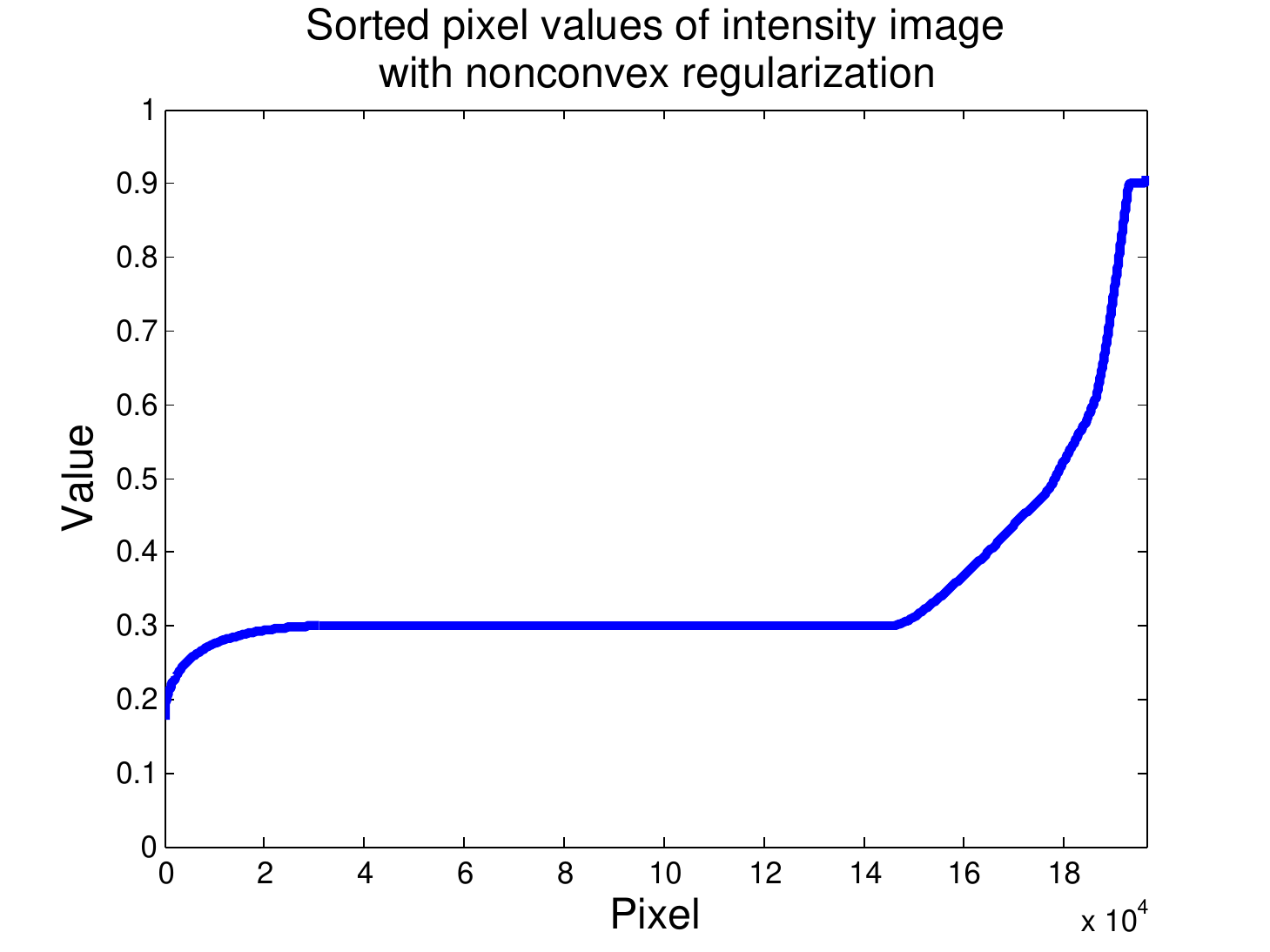}} 
    \caption{Example of incorporating a nonconvex tonemapping term into the ROF model.}
    \label{fig:denoiseAndIllucorrect}
\end{figure}
}

\newcommand\figIllucorrectionTwo
{
\begin{figure}[t!]
  \begin{center}
    \subfigure[{Noisy image with small range of values.}]{\includegraphics[width=0.32\textwidth]{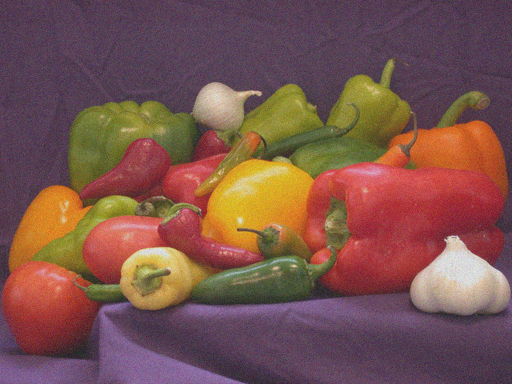}}
    \subfigure[{TV denoised image.}]{\includegraphics[width=0.32\textwidth, natwidth=560,natheight=420]{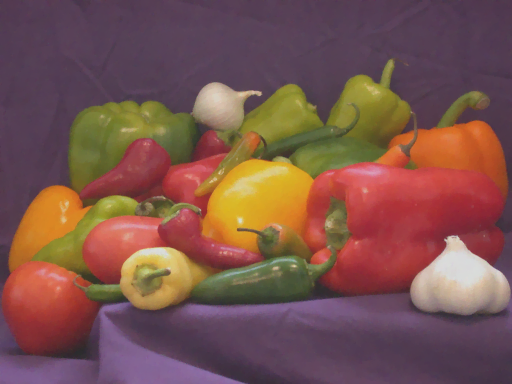}}
    \subfigure[Image after TV denoising along with encouraging intensity values to be either 0 or 1.]{\includegraphics[width=0.32\textwidth, natwidth=560,natheight=420]{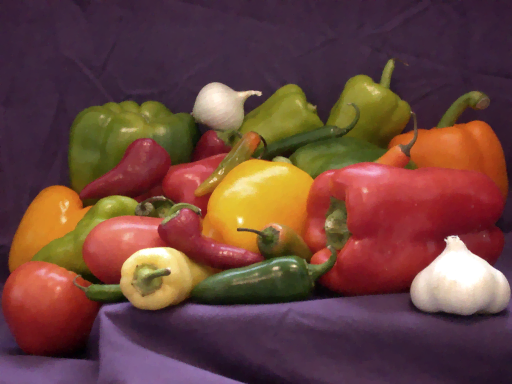}} 
    \caption{Example of incorporating a nonconvex tonemapping term into the ROF model.}
    \label{fig:denoiseAndIllucorrect2}
  \end{center}
\end{figure}
}

\title{The Primal-Dual Hybrid Gradient Method for Semiconvex Splittings}

\author{Thomas M\"ollenhoff \thanks{Technical University Munich, Boltzmannstrasse 3,
85748 Garching, Germany. This work was supported by the ERC Starting Grant 'Convex Vision'.} \and Evgeny Strekalovskiy \footnotemark[1] \and Michael Moeller \footnotemark[1] \and Daniel Cremers \footnotemark[1]}
\begin{document}
\maketitle

\begin{abstract}
  This paper deals with the analysis of 
 a recent reformulation of the primal-dual hybrid gradient method
  \cite{ZhuChan08, Pock-et-al-iccv09, Esser-Zhang-Chan-10, Chambolle-Pock-jmiv11}, which allows to apply it to nonconvex
  regularizers as first proposed for truncated quadratic penalization
  in \cite{MS-tr}. 
Particularly, it investigates
  variational problems for which the energy to be minimized can be
  written as $G(u) + F(Ku)$, where $G$ is convex, $F$ semiconvex, and $K$ is a linear operator. 
  We study the method and prove convergence in the case
  where the nonconvexity of $F$ is compensated by the strong convexity
  of the $G$. The convergence proof yields an interesting requirement
  for the choice of algorithm parameters, which we show to not only be
  sufficient, but necessary. Additionally, we show boundedness of the
  iterates under much weaker conditions. Finally, we demonstrate
  effectiveness and convergence of the algorithm beyond the
  theoretical guarantees in several numerical experiments.
\end{abstract}

\section{Introduction}
One of the most successful and popular approaches to ill-posed inverse
problems particularly those arising in imaging and computer vision are
variational methods. Typically, one designs an energy functional $E: X
\rightarrow \Rext$, which maps an image $u$ from a Banach space $X$ to
a number on the extended real value line, such that a low energy
corresponds to an image with desired properties. The solution is then
determined by finding the argument that minimizes the corresponding
energy functional. In this paper we are concerned with energies that
can be written as
$$E(u) = G(u) + F(Ku).$$ 
In many cases $G(u)$ can be interpreted as a data fidelity term
enforcing some kind of closeness to the measurements and $F(Ku)$
serves as a regularization term which incorporates prior information
about $Ku$, where $K$ denotes a linear operator.

One property that many functionals appearing in practice have in
common is that both $G$ and $F$ often have easy-to-evaluate proximity
operators, i.e. the minimization problem
$$ \rslv{\tau \partial G}(v)=\underset{u}\argmin G(u) + \frac{1}{2 \tau}\|u-v\|_2^2 $$ 
can be solved efficiently to a high precision or even has a closed
form solution. This observation has lead to many very efficient
optimization methods, such as the primal-dual hybrid gradient method
(PDHG), the alternating minimization algorithm (AMA) or equivalently
proximal forward backward splitting (PFBS) or the alternating
directions method of multipliers (ADMM).
We refer to \cite{Esser-Zhang-Chan-10} for a detailed overview over these different splitting methods and their connections. 

Although the majority of the proposed variational approaches in image
processing and computer vision have focused on the use of convex
functionals, there are many interesting nonconvex functionals such as
the Mumford-Shah functional \cite{Mumford-Shah-89} or related
piecewise smooth approximations \cite{Blake-Zisserman-87,
Morel-Solimini-95}, nonconvex regularizers such as TV$^q$ approaches
motivated by natural image statistics (cf.~\cite{Huang-Mumford-99,
HyperLaplacian}), nonconvex constraints such as orthogonality
constraints or sperical constraints (cf.~\cite{Wen-Yin-13} and the
references therein), or application/model dependent nonconvex
functionals for which the prior knowledge cannot be appropriately
represented in a convex framework
(e.g. \cite{Esser-Zhang-14}). Interestingly, the proximity operators
to the nonconvex functions can often still be evaluated efficiently,
such that splitting approaches seem to be a very effective choice even
in the nonconvex case.

This paper investigates the use of a particular splitting based
optimization technique for certain nonconvex $F$ with easy-to-evaluate
proximity operators. As we will see, the proposed algorithm coincides
with the PDHG method in the convex case, but is reformulated such that
it can be applied in a straightforward fashion even to energies with
nonconvex regularization functionals. This reformulation was initially
proposed in \cite{MS-tr}.

The rest of this paper is organized as follows. First, we will summarize and discuss the literature on nonconvex optimization and point out some difficulties of existing methods. 
Particularly, we discuss why many methods are rather costly in
comparison to direct splitting approaches, such as our proposed
iterative PDHG scheme. In Section~\ref{sec:theory} we will analyze the
behavior of the algorithm and prove that it converges to the minimum
in case where the nonconvexity of $F$ is compensated by the strong
convexity of $G$. Additionally, we provide an example showing that the
parameter choices which are sufficient to guarantee the convergence
are also necessary. Furthermore, we will follow the analysis done for
the same algorithm in the case of a nonconvex $G$ but convex $F$ in
\cite{Esser-Zhang-14}, and show that the iterates remain bounded under
rather mild conditions. We can conclude the convergence to a critical
point if additionally the difference between two sucessive iterates
tends to zero. In our numerical experiments in
Section~\ref{sec:experiments} we will first discuss some cases in
which our full convergence proof holds and illustrate why this class
of problems can be efficiently tackled with splitting
approaches. Next, we provide numerical examples showing that
convergence along subsequences to critical points is achieved for a
much wider class of problems and often goes beyond the theoretical
guarantees. Finally, we draw conclusions and point out future
directions of research in Section~\ref{sec:conclusions}.

\subsection{Related Work}
For the rest of this paper, let $X$ and $Y$ be finite dimensional Hilbert spaces endowed with an
inner product $\iprod{\cdot}{\cdot}$ and induced norm $\norm{\cdot}$. In this paper we are concerned with the following class of
problems.Let $K
: X \rightarrow Y$ be a linear operator and consider the optimization problem
\begin{equation}
  \underset{u} \min \enskip G(u) + F(K u),
  \tag{P}
  \label{eq:optiProb}
\end{equation}
where $G : X \rightarrow \Rext$ is proper, convex and lower semicontinuous, $F : Y \rightarrow \Rext$ proper, lower semicontinuous and possibly nonconvex. For now we will not specify further properties of $F$, but in the absence of convexity one typically needs some kind of smoothness assumption or, as in our case, $\omega$-semiconvexity as defined in Section~\ref{sec:theory}. We will also work with the equivalent problem
\begin{equation}
  \underset{u,g} \min \enskip G(u) + F(g) \quad \text{subject to} \quad g = Ku.
  \tag{PC}
  \label{eq:optiProbConstrained}
\end{equation}
We will also assume that the above minimization problems are well defined in the sense that minimizers exist.
 
Generally, problems of the form \eqref{eq:optiProb} or even more general nonconvex minimization problems have been studied in several cases and there exist a number of optimization techniques, which provably converge to critical points. In the following we will briefly summarize these approaches:
\begin{itemize}

\item The simplest and probably oldest approach is to use gradient
  descent in case where the entire energy is smooth. If $E$ is coercive,
  and one uses gradient descent with appropriate step sizes, the
  sequence is bounded and every accumulation point is critical
  (cf.~\cite{Bertsekas-2003}). Gradient descent methods are, however,
  often inefficient. Additionally, they require the whole energy to be
  smooth and thus cannot deal with hard constraints.

\item In the case where the nonconvex part is smooth and the convex
  part is nondifferentiable, one can apply forward-backward splitting
  techniques. We refer to \cite{attouch13} and the references therein
  for convergence results for descent methods. To give an example,
  recently, Ochs et al. \cite{iPiano} proposed a forward-backward splitting approach
  called \textit{iPiano} which additionally uses a heavy-ball type of
  approach and provably converges to critical points under mild
  conditions. The basic idea of forward-backward
  splitting approaches is to construct a sequence of minimizers
  $$ u^{k+1} = \rslv{\tau_k \partial G}\left( u^k - \tau_k K^T \nabla F \left( Ku^k \right) \right).$$ 
  While this approach works well and is applicable to many situations it
  has the drawback that $F$ has to be differentiable. Thus any type of
  hard constraint has to be incorporated into $G$, which can make the above
  minimization problem hard to evaluate if no efficient closed form
  projection exists.


\item For nonconvex functions which just need an additional squared
  $\ell^2$ norm to become convex (also called $\omega$-semiconvex
  functions as defined in Section \ref{sec:theory}), Artina, Fornasier
  and Solombrino proposed an Augmented Lagrangian based scheme for
  constrained problems and proved convergence to critical points
  \cite{fornasier}. 



  The drawback of that approach, however, would be that the resulting
  scheme requires the solution of an inner convex problem at each
  iteration which might become costly.

\item Along the same lines, there exist approaches in case the
  functional can be written as the difference of two convex
  functions. For instance, if $F(Ku) + \frac{\omega}{2}\|Ku\|_2^2$ is
  convex, then the difference of convex functions algorithm (DCA, see
  e.g. \cite[Algorithm 1]{phamdinh}) would be
  \begin{align*} 
    u_{n+1} = \underset{u} \argmin G(u) + F(Ku) +
    \frac{\omega}{2}\|Ku\|_2^2 - \omega\langle K^TKu^n, u \rangle.
  \end{align*}

  In this case the convergence to a critical point can be proven under
  mild conditions (cf.~\cite{An-Tao-05,phamdinh}). Once more, the scheme
  requires the solution of a convex optimization problem at each
  iteration and will therefore be rather costly. Often times, the problem
  with these approaches is that it is not clear how many iterations are
  required to solve the inner convex problem.

\item If $G$ is convex and $F$ is well approximated from above by a
  quadratic or a $\ell^1$ norm, there exist the iteratively reweighted least squares
  and iteratively reweighted $\ell^1$ algorithms. These methods solve a sequence
  of convex problems which majorize the original nonconvex cost function. The latter method has been recently
  generalized by Ochs et. al.~\cite{IRL1} to handle linearly constrained
  nonconvex problems which can be written as the sum of a convex and a concave function. Furthermore they prove convergence to a 
  stationary point if the concave term is smooth. While in principle it is still required to solve
  a sequence of convex subproblems, they precisely specify to which accuracy
  each subproblem has to be solved which makes their algorithm very efficient in practice.

\item Another track of algorithms are alternating minimization methods which
  are based on the well-known quadratic penalty method. In the classic
  quadratic penalty method the constrained cost function is augmented by a quadratic term
  that penalizes constraint violation in order to arrive at an
  unconstrained optimization problem:
  \begin{equation}
    \underset{u, g} \min \enskip G(u) + F(g) + \frac{\tau}{2} \|g - K u\|_2^2.
    \label{eq:problem_general_qd}
  \end{equation}
  This augmented cost function is then solved for increasing values of
  $\tau \rightarrow \infty$, see e.g. \cite{nikolova10}. Since joint
  minimization over $u$ and $g$ is difficult, minimization is usually
  carried out in an alternating fashion. Variations exist where the
  alternating minimizations are repeated several times before increasing
  the penalty parameter $\tau$, see e.g. \cite{HyperLaplacian} for
  details. We refer the reader to \cite{attouch10} and the references
  therein for a discussion on the convergence analysis for this class of
  methods. 

\item Finally, there exist some theory for the convergence of the
  PDHG method in the cases that the nonconvexity can be isolated in an
  operator with certain properties \cite{Valkonen-14}, which is
  different from the problem we are investigating in this paper.
\end{itemize}
The above approaches that can deal with constraints or other
nondifferentiabilities all have the disadvantage, that each step of
the minimization algorithm involves the solution of a convex problem,
and hence each step is as expensive as a full typical minimization
algorithm such as the PDHG method. This is the reason why some authors
have started applying popular convex optimization methods mentioned in
the previous section to nonconvex problems as if they were convex
(e.g. \cite{Chartrand-13, Chartrand-12, Esser-Zhang-14, Ozolins14}). Although the
results obtained in practise are convincing, very little is known
about the convergence of these methods. One theoretical
result we would like to point out is the one by Zhang and Esser in
\cite{Esser-Zhang-14}, where boundedness results and convergence under
an additional (iteration dependend) assumption is shown. While
\cite{Esser-Zhang-14} investigated the use of the PDHG method for a
particular type of nonconvex $G$ and convex $F$, we will reproduce and
extend their results for nonconvex $F$ and convex $G$.

The algorithm we are investigating has been initially proposed by
Strekalovskiy and Cremers \cite{MS-tr} and is itself is a reformulation
of \cite{ZhuChan08, Pock-et-al-iccv09, Chambolle-Pock-jmiv11} to make
it applicable to nonconvex $F$. A detailed description as well as a
derivation and its relation to the PDHG method is given in the next
section.

\subsection{The Algorithm}
We study the following algorithm that aims at minimizing \eqref{eq:optiProb} for $\omega$-semiconvex $F$ with $\omega > 0$ as defined in Section~\ref{sec:theory}. Given an $(u^0, q^0) \in X \times Y$ and for $\bar u^0 = u^0$, $\sigma \geq 2 \omega$, $\tau \sigma \norm{K}^2 \leq 1$, $\tau > 0$, $\sigma > 0$, $\theta \in [0,1]$, iterate for all $n \geq 0$:
\begin{equation}
  \begin{aligned}
    g^{n+1} &= \underset{g} \argmin \frac{\sigma}{2}\|g - K\bar{u}^n\|^2_2  - \langle g, q^n \rangle + F(g),\\
    q^{n+1} &= q^n + \sigma(K\bar{u}^n - g^{n+1}), \\
    u^{n+1} &= \underset{u} \argmin \frac{1}{2\tau} \|u - u^n\|_2^2 + \langle Ku, q^{n+1}\rangle + G(u), \\
    \bar{u}^{n+1} &= u^{n+1} + \theta(u^{n+1} - u^n).
  \end{aligned}
  \label{eq:firstFormulation}
\end{equation}
For the above update scheme to be well defined in the absence of convexity, we need
existence and uniqueness of a minimizer for the minimization problem in $g$. We will see that
this is fulfilled for $\omega$-semiconvex $F$ and for all $\sigma > \omega$. 
Note that the unusual step size criterion on $\sigma$, i.e. $\sigma\geq 2\omega$, will follow from our main theoretical convergence result in Section~\ref{sec:theory}. 

For \emph{convex} $F$, the above algorithm is exactly the primal-dual
hybrid gradient method as proposed in 
\cite{Pock-et-al-iccv09,Chambolle-Pock-jmiv11}. If we stack the primal
and dual variable $z^n := (u^n, \enskip q^n)^T$ and start the iteration
scheme with the primal variable update (which is equivalent to
\eqref{eq:firstFormulation} for a different initial value $u^0$) our
reformulation takes the following simple form:
\begin{equation}
  0 \in T(z^{n+1}) + \M (z^{n+1} - z^{n})
  \label{eq:compactFormulation}
\end{equation}
with
\begin{equation}
  T = \begin{pmatrix}
    \partial G & K^T \\
    -K & (\partial F)^{-1}
  \end{pmatrix}, \qquad
  \M = \begin{pmatrix}
    \tau^{-1} I & -K^T \\
    -\theta K & \sigma^{-1} I
  \end{pmatrix},
\end{equation}
where $(\partial F)^{-1}$ denotes the inverse of the set-valued
operator $\partial F$. To see this, we use the identity $(\partial F)^{-1} = (\sigma I - \sigma (I +
\sigma^{-1} \partial F)^{-1})^{-1} - \sigma^{-1} I$ (see
\cite[Prop. 23.6 (ii)]{BauschkeCombettesBook}), and writing out
\eqref{eq:compactFormulation} with that, we again obtain the algorithm
\begin{equation}
  \begin{aligned}
    u^{n+1} &= \rslv{\tau \partial G} \left(u^n - \tau K^T q^n \right), \\
    q^{n+1} &=  \rslv{\sigma (\partial F)^{-1}} \left(q^n + \sigma K (2 u^{n+1} - u^n) \right) \\
          &= q^n + \sigma \bigl( K \left( 2 u^{n+1} - u^n \right) - \underbrace{\rslv{\sigma^{-1} \partial F} \left(K \left(2 u^{n+1} - u^n \right) + \sigma^{-1} q^n \right)}_{=:g^{n+1}} \bigr) ,
  \end{aligned}
\end{equation}
which is exactly \eqref{eq:firstFormulation} for $\theta=1$ and with the order of primal and dual update reversed.

If $F$ is proper, convex and lower semicontinuous it
readily follows that $(\partial F)^{-1} = \partial F^*$ and we arrive at the
iteration scheme described in \cite{Pock-Chambolle-iccv11}. In that
case it is easy to check that $T$ is a maximally monotone operator,
and since for $\theta = 1$ the matrix $\M$ is symmetric and positive
definite \cite[Remark 1]{Pock-Chambolle-iccv11} it is well
known that the algorithm is just a special case of the classical
proximal point algorithm in a different norm. Thus it converges to a
saddle point $\widehat z \in X \times Y$ with $0 \in T \, \widehat z$.

In that sense we can interpret $(\partial F)^{-1}$ as a natural
generalization of the subdifferential of the convex conjugate
$\partial F^*$ to nonconvex functions. However, for general
nonconvex $F$ the operator $T$ fails to be monotone, and thus the
convergence analysis becomes more complicated. There exist
convergence results for the proximal point algorithm for nonmonotone
operators, in particular for operators whose inverse $T^{-1}$ is
$\rho$-hypomonotone \cite{proxNonmonotone1,proxNonmonotone2}. While for certain choices
of $F$ and $G$ we can show that $T^{-1}$ is $\rho$-hypomonotone, the
results obtained by this analysis are much weaker than the ones
provided in the rest of this paper. Analyizing the algorithm from this
set-valued operator point of view is left for further future work.

Additionally, notice that for $\theta=0$, the primal-dual algorithm
\eqref{eq:firstFormulation} is remarkably similar to the alternating
direction method of multipliers (ADMM) which we obtain by replacing
$\frac{1}{2 \tau} \| u - u^n\|^2$ by $\frac{\sigma}{2}\|Ku -
g^{n+1}\|^2$ in the update for $u^{n+1}$. The updates for $g^{n+1}$
and $q^{n+1}$ are exactly the same. Connections of this type have been
pointed out in \cite{Esser-Zhang-Chan-10} where the method we are
considering here was interpreted in the sense of a preconditioned ADMM
algorithm. Further generalizations for simplifying subproblems were
investigated in \cite{Deng-Yin-2014}.

\subsection{Contribution}
Our contribution in this paper is to present a generalized usage of
the algorithm proposed in \cite{MS-tr} and study its convergence
behavior. We can prove the convergence for a certain choice of
parameters in case the strong convexity of $G$ dominates the
semiconvexity of the other term $F$. We find an simple sufficient
requirement for the step size parameter, namely that the step size for
the minimization in $g$ has to be twice as large as necessary to make
the subproblem convex and show that this criterion is also necessary by
constructing an example for a diverging algorithm (for an overall
convex energy) in case the step size requirement is violated. In a
more general setting we can state the boundedness of the iterates
under mild conditions, and can show a result similar to the one
obtained in \cite{Esser-Zhang-14} for nonconvex functions ~$G$. 
Finally, we demonstrate the efficiency of the proposed method in
several numerical examples which show that the convergence of the
algorithm often goes beyond the theoretical guarantees.

\section{Theory}
\label{sec:theory}
\subsection{Mathematical preliminaries}
Before going into the analysis of the algorithm, let us recall some definitions that give sense to the proposed algorithm and the following analysis. 

\begin{defi}[Effective domain]
  Let $E : X \rightarrow \Rext$. We call the set
  $$ \text{dom}(E) := \{ u \in X ~|~ E(u) < \infty \}$$
  the effective domain of $E$. If $\text{dom}(E)$ is 
  a nonempty set, then we call $E$ proper.
\end{defi}

\begin{defi}[{Subdifferential, cf.~\cite[Definition 8.3]{Rockafellar-Variational-Analysis}}]
  \label{def:subdiff}
  Let $X$ be a finite dimensional Hilbert space and $E: X \rightarrow \Rext$ a lower semicontinuous functional. We say that $\xi \in X$ belongs to the (regular) subdifferential $\partial E(g)$ of $E$ at a point $g \in \text{dom}(E)$ if and only if 
  $$ \underset{\substack{\rho \rightarrow g \\ \rho \neq g} }{\lim \inf} \enskip \frac{E(\rho) - E(g) - \langle \xi, \rho - g \rangle}{\|\rho-g\|} \geq 0. $$
  For $g \not\in \text{dom}(E)$ we define $\partial E(g) = \emptyset$.
\end{defi}
It is well known that this definition extends (and is consistent with) the usual definition of the subdifferential for convex functions:
$$ \partial E(g) = \{ \xi \in X ~|~ E(\rho) - E(g) - \langle \xi, \rho - g \rangle \geq 0 \}, ~~~ g \in \text{dom}(E).$$
For illustrative purposes Fig.~\ref{fig:1dSubgrads} shows some examples of subgradients for one dimensional functions. 

\figSubdifferential

\begin{rem}[{Sum of subdifferentials, from \cite{fornasier}, \cite[Exercise 8.8(c)]{Rockafellar-Variational-Analysis}}]
  \label{rem:sumOfSub}
  We say that $E: X \rightarrow \Rext$ is a $C^1$ perturbation of a
  lower semicontinuous function, if $E = \widehat F + Q$ for $\widehat F : X \rightarrow
  \Rext$ being lower semicontinuous and $Q : X \rightarrow \R$ being of
  class $C^1$. For this type of functions the decomposition
  $$\partial E(u) = \partial \widehat F(u) + \partial Q(u) $$
  holds true with the subdifferential $\partial$ from Definition~\ref{def:subdiff}.
\end{rem}

\begin{defi}[Critical point]
  Let $E: X \rightarrow \Rext$ be a lower semicontinuous function. We say that $u \in X$ is a critical point if $0 \in \partial E(u)$. 
\end{defi}

\begin{defi}[Semiconvexity and strong convexity, from \cite{fornasier}]$ $ \\ \vspace{-0.5cm}
  \begin{itemize}
  \item We call a lower semicontinuous function $E: X \rightarrow \Rext$ $\omega$-semiconvex if $E + \frac{\omega}{2} \| \cdot \|^2$ is convex. 
  \item We call a lower semicontinuous function $E: X \rightarrow \Rext$ $c$-strongly convex, if for all $u,v \in H$, $q \in \partial E(u)$, $\xi \in \partial E(v)$, it holds that
    $$\langle u-v, q - \xi \rangle \geq c \|u-v\|^2.$$
  \end{itemize}
\end{defi}
One typically uses the term $c$-strongly convex for $c>0$ and the term $\omega$-semiconvex for $\omega \geq 0$, but if one allows negative values for the constants, it is well-known that both definitions can be used to express the same thing.
\begin{rem}
  \label{lem:equivalenceConvexityTypes}
  An $\omega$-semiconvex function is $c$-strongly convex with $c = -\omega$. A $c$-strongly convex function is $\omega$-semiconvex with $\omega = -c$. For example, $u \mapsto \frac{c}{2} \norm{u}^2$ is c-strongly convex and $y \mapsto -\frac{\omega}{2} \norm{y}^2$ is $\omega$-semiconvex.
\end{rem}
%

\begin{lemma}
  Let $E(u) = (F \circ K)(u) + G(u)$ be the sum of a proper, lower semicontinuous $\omega$-semiconvex function $F \circ K$ (with $K : X \rightarrow Y$ being a linear operator) and a proper, lower semicontinuous convex function $G$. Additionally, let $\text{dom}(F \circ K) \cap \text{dom}(G) \neq \emptyset$. Then the following holds true
  $$ \partial E(u) = K^T (\partial F)(Ku) + \partial G(u). $$
  \label{lem:sumOfSubdiffGeneral}
\end{lemma}
\begin{proof}
  Due to Remark \ref{rem:sumOfSub}, as well as the result that the subdifferential of the sum of two convex functions is the sum of the subdifferentials if their domains have a nonempty intersection (cf.~\cite{Rockafellar}), we find that 
  \begin{equation*}
    \begin{aligned}
      \partial \left( G(u) + F(Ku) \right) &= \partial \left( G(u) + F(Ku) + \frac{\omega}{2}\|u\|^2\right) - \omega u\\
      &= \partial G(u) +  \partial \left(F(Ku) + \frac{\omega}{2}\|u\|^2\right) - \omega u\\
      &= \partial G(u) +  K^T (\partial F)(Ku).
    \end{aligned}
  \end{equation*}

\end{proof}
\subsection{Convergence Analysis}
Throughout the rest of the paper, we will assume that the assumptions of Lemma \ref{lem:sumOfSubdiffGeneral} hold. Particularly, we make the stronger assumption that $F$ itself is $\omega$-semiconvex, such that the subproblem in $g$ arising in \eqref{eq:firstFormulation} is strongly convex for $\sigma > \omega$ and the iterates are well defined. Thus, let $\sigma > \omega$ hold for the rest of this work. With the tools defined in the previous section, we can now state the optimality conditions for Algorithm \eqref{eq:firstFormulation} to be
\begin{align}
  0 &\in \sigma(g^{n+1}- K\bar{u}^n) + \partial F(g^{n+1}) - q^n, \\
  0 &\in \frac{1}{\tau}(u^{n+1} - u^n) + K^T q^{n+1} + \partial G(u^{n+1}).
\end{align}
The optimality conditions yield an interesting property of $q^n$. Using the definition of $q^{n}$ in the optimality condition for $g^{n}$ leads to  
\begin{equation}
  \label{eq:qIsSubgradient}
  q^n \in \partial F(g^n) \qquad \ \forall n \geq 1,
\end{equation}
such that the variable $q^n$ has an immediate interpretation. Notice that $q^n$ is one of the two main variables in the PDHG method in the convex case, where it is typically updated with a proximity operator involving the convex conjugate $F^*$. Formulating the algorithm in terms of an update of the primal variable $g$ allows us to apply the same iterative scheme in the nonconvex case. Notice that $F^*$ is still defined in the case where $F$ is nonconvex, however, $F^*$ will be convex independent of the convexity of $F$. Using a PDHG formulation involving $F^*$ would therefore implicitly use convex relaxation which is not always desirable. 



\subsubsection{Convergence Analysis: Semiconvex $+$ strongly convex}

In this subsection we will investigate the convergence in the special case, where $F$ is semi-convex, $G$ is strongly convex, and the strong convexity dominates the semiconvexity. Our main result is stated in the following theorem. For the sake of clarity and readability, we moved the corresponding proof to Appendix~\ref{lbl:appendixProof}. 

\begin{thm}
  \label{thm:convergence}
  If $F$ is $\omega$-semiconvex and $G$ $c$-strongly convex with a constant $c > \omega \|K\|^2$, then Algorithm \eqref{eq:firstFormulation} converges to the unique solution $\hat{u} = \arg \min_{\,u} \, G(u) + F(Ku)$ of \eqref{eq:optiProb} for $\sigma = 2 \omega$, $\tau \sigma \|K\|^2 \leq 1$, and any $\theta \in [0,1]$, with rate $\|u^n - \hat{u}\|^2 \leq C/n$.
\end{thm}

Although the requirements of the theorem lead to the overall energy being strongly convex, this is to the best knowledge of the authors the first convergence result for an algorithm where a nonconvex part of the energy was separated from a convex one by variable splitting. We will show some examples for energies our above convergence result can be applied to in the numerical results Section~\ref{sec:experiments}. Regarding the Theorem~\ref{thm:convergence}, we'd like to remark the following:
\begin{rem}
  \label{rem:gConvergence}
  Theorem \ref{thm:convergence} states the convergence of the variable $u$ only, which is due to the choice $\sigma = 2 \omega$ which does not necessarily lead to a convergence for $g$ as we will see in Example~\ref{exm:divergence} below. Let us note that, however, we can modify Theorem \ref{thm:convergence} to also yield convergence of $g$ by choosing $\sigma > 2 \omega$, but small enough such that $c > \frac{\sigma \|K\|^2}{2}$. In the very first estimate of the proof, this leads to having terms of both forms, $\|u^{n+1}-\hat{u}\|^2$ as well as $\|g^{n+1}-\hat{g}\|^2$, with strictly positive factors. The remaining proof continues as before and one can conclude the convergence of $\|g^{n+1}-\hat{g}\|^2 \rightarrow 0$ at least as fast as $1/n$ as well. 
\end{rem}

Notice that the assumptions in Theorem \ref{thm:convergence} require $\sigma = 2 \omega$, i.e. $\sigma$ needs to be chosen twice as large as necessary to make the subproblem for the minimization in $g$ convex. It is interesting to see that this requirement is not based on a (possibly crude) estimate for the convergence. In fact, it is not only sufficient but also necessary, as the next proposition shows.
\begin{prop}
  \label{prop:divergenceExample}
  Let all assumptions for Theorem \ref{thm:convergence} be met, except that we choose $\sigma < 2 \omega$. Then there exists a problem for which the sequences of $q^n$ and $g^n$ of the proposed algorithm with $\theta = 0$ diverge for any fixed $\tau > 0$.
\end{prop}
\begin{proof}
  Let us give a very simple example problem for which the algorithm diverges for $\sigma < 2 \omega$. Consider
  $$ \hat{u} = \arg \min_u ~ \frac{c}{2}\|u\|^2 - \frac{1}{2}\|u\|^2, $$
  Clearly, for $c > 1$ the problem is strongly convex and the solution is given by $\hat{u}=0$. We apply the proposed algorithm with $F(Ku) = - \frac{1}{2}\|u\|^2$, $K$ being the identity, $\|K\|=1$, and $G(u) = \frac{c}{2} \|u\|^2$ being $c$ strongly convex. We assume $\sigma < 2 \omega =2$, and can immediately see that $\sigma >1$ has to hold for the subproblem in $g$ to even have a minimizer, i.e. for the algorithm to be well defined. Note that $q^n \in \partial F(g^n) = \{ -g^n \}$ by \eqref{eq:qIsSubgradient} such that we can eliminate $q^n$ from the algorithm; a short computation shows that
  \begin{align*}
    g^{n+1} &= -\frac{1}{\sigma - 1} g^n + \frac{\sigma}{\sigma-1}u^n, \\
    u^{n+1} &= \frac{1}{\tau^{-1}+c}g^{n+1} + \frac{\tau^{-1}}{\tau^{-1}+c}u^n.
  \end{align*}
  Using the first formula to replace $g^{n+1}$ in the second equation admits the fixed point form
  \begin{align*}
    \left(
      \begin{array}{c}
        g^{n+1} \\
        u^{n+1}
      \end{array}
    \right) = 
    \left(
      \begin{array}{cc}
        a_{1,1} & a_{1,2} \\
        a_{2,1}  & a_{2,2}
      \end{array}
    \right)
    \left(
      \begin{array}{c}
        g^{n} \\
        u^{n}
      \end{array}
    \right).
  \end{align*}
  with
  \begin{align*}
    a_{1,1} =& \  -\frac{1}{\sigma - 1} \ \ \qquad \qquad 
    a_{1,2} = \  \frac{\sigma}{\sigma-1}\\
    a_{2,1} =&  \ - \frac{1}{(\tau^{-1}+c) (\sigma -1)}\ \ \
    a_{2,2} =  \ \frac{\tau^{-1}}{\tau^{-1}+c} + \frac{\sigma}{(\tau^{-1}+c) (\sigma -1)}
  \end{align*}
  Obviously, this iteration can diverge as soon as one eigenvalue of the matrix $A$ is larger than one -- simply choose the corresponding eigenvector as a starting point of the iteration. The eigenvalues of a $2\times 2$ matrix are given by
  \begin{align*}
    d_{1} =& \frac{a_{1,1}+a_{2,2}}{2} +  \sqrt{\left(\frac{a_{1,1}+a_{2,2}}{2}\right)^2 - (a_{1,1} a_{2,2} - a_{1,2}a_{2,1})} \\
    d_{2} =& \frac{a_{1,1}+a_{2,2}}{2} - \sqrt{\left(\frac{a_{1,1}+a_{2,2}}{2}\right)^2 - (a_{1,1} a_{2,2} - a_{1,2}a_{2,1})} 
  \end{align*}
  Due to $- (a_{1,1} a_{2,2} - a_{1,2}a_{2,1}) = \frac{\tau^{-1}}{(\sigma-1)(\tau^{-1}+c)}$, we know that $d_1$ and $d_2$ are real. Note that $a_{2,1} \stackrel{c \rightarrow \infty}{\rightarrow} 0$ and $a_{2,2} \stackrel{c \rightarrow \infty}{\rightarrow} 0$, such that we could choose $c$ large enough such that 
  $$d_2 = a_{1,1} + \varepsilon(c),$$
  with $\varepsilon(c)$ arbitrarily small. Assume we have picked some $\sigma<2$. Then $a_{1,1}<-1$ and we can pick a $c$ large enough such that $d_2<-1$ and the algorithm diverges. Notice that for any finite $c$, we have $a_{2,1}\neq 0$ and $a_{1,2} \neq 0$ such that neither $(1,0)^T$ nor $(0,1)^T$ are eigenvectors which means that both variables, $u^n$ and $g^n$, diverge.

\end{proof}
Note that in the above example, the algorithm can diverge although the total energy was strongly convex, which shows that the convergence result of Theorem \ref{thm:convergence} is not trivial. Additionally, we could see that choosing a large $\alpha$, i.e. making the total energy even more strongly convex, led to an algorithm that provably diverges. As a second example, let us also state that the case of $u^n$ converging without $g^n$ converging can also happen in practice and is not an artifact of the convergence estimates:
\begin{example}
  \label{exm:divergence}
  Consider
  $$ \hat{u} = \arg \min_u ~ \frac{c}{2}\|u\|^2 - \frac{1}{2}\|Ku\|^2, $$
  with $K= (1,1)^T$. Clearly, for $c > 2$ the problem is strongly convex and the solution is given by $\hat{u}=0$. If we apply the proposed algorithm with $F(Ku) = - \frac{1}{2}\|Ku\|^2$, $\theta = 0$, and start with $u^0 = 0$, $q^0 = (-1, 1)$, we obtain
  \begin{align*}
    g^{n+1} &= (\sigma - 1)^{-1} q^n, \\
    q^{n+1} &= - g^{n+1},\\
    u^{n+1} &= 0.
  \end{align*}
  Notice that we only need $\sigma \geq 1$ for the minimization problem in $g$ to be convex and $\sigma > 1$ for the problem to be strictly convex, coercive and the iterates to be well defined. However, $F(g)$ is $\omega$-semi convex with $\omega=1$, and Theorem \ref{thm:convergence} tells us that we need $\sigma \geq  2\omega$ for the convergence. Looking at the above iteration we see that indeed $g^{n+1} = (-1)^{n} (\sigma - 1)^{-(n+1)} q^0$ such that $g^n$ diverges for $\sigma  <2$. It is interesting to see that Theorem \ref{thm:convergence} states the convergence of the variable $u$ only, which is due to the choice $\sigma = 2 \omega$ which - as we can see here - does not need to lead to a convergence for $g$. As pointed out in remark \ref{rem:gConvergence} we need to choose a slightly larger $\sigma$ to also obtain the convergence in $g$, which is consistent with the behavior of the algorithm in this toy example. 

  In any case, this example as well as the previous proposition show, that the estimates used to prove the convergence of the algorithm seem to be sharp - at least in terms of the algorithm parameters. 
\end{example}

Although the assumptions in Theorem \ref{thm:convergence} seem to be rather restrictive it is interesting to see that there are some nontrivial optimization problems which meet the necessary requirements as we will see in the numerical experiments section. 

\subsubsection{Convergence Analysis: Discussion in the nonconvex case}
In the case where the energy is truly nonconvex, the situation is much more difficult and we do not have a full convergence result yet. Based on the example of the diverging algorithm in Proposition~\ref{prop:divergenceExample} we can already see that we will need some additional assumptions. Particularly, it seems that we need the update of $g$ to be somehow contractive - despite the nonconvexity. Still, there are a few things one can state for fairly general nonconvex functions. The following analysis is closely related to the analysis in \cite{Esser-Zhang-14}, where the primal-dual algorithm for nonconvex data fidelity term ($G$ in our notation) was analyzed. 

\begin{prop}
  \label{prop:boundedness}
  Let $F$ be a differentiable function whose derivative is uniformly
  bounded by some constant $b \geq \| \nabla F(g) \|$ for all
  $g$. Additionally, let there exist constants $a \in \mathbb{R}^+$ and
  $t\in \mathbb{R}^+$ with $t> b\|K\|$, such that for all $u$ with
  $\|u\|\geq a$, we have
  $$ G((1+\varepsilon)u) \geq G(u) + t\varepsilon \|u\|, \qquad \forall \varepsilon \geq 0.$$
  Then the sequence $(u^n, q^n, g^n)$ is bounded and therefore converges along subsequences. 
\end{prop}
\begin{proof}
  Since $q^n = \nabla F(g^n)$ (see \eqref{eq:qIsSubgradient}), the boundedness of $q^n$ is trivial. We prove the boundedness of $u^n$ by contradiction. Assume there exists a $u^{n+1}$ with $\|u^{n+1}\|>a$. Let us pick the first iterate with this property (assuming $\|u^0\|<a$), which ensures that $\|u^n\|\leq a$. We know that $u^{n+1}$ minimizes 
  $$ E(u) = \frac{1}{2 \tau}\|u-u^n\|^2 + G(u) + \langle u, K^T \nabla F(g^{n+1})\rangle. $$
  Consider $\tilde{u}(\varepsilon) = \frac{1}{1+\varepsilon} u^{n+1}$. We will show that there exists a $ \tilde{u}(\varepsilon)$ which has a lower energy than $u^{n+1}$, thus leading to a contradiction. One finds that
  \begin{align*}
    G(u^{n+1}) = G((1+\varepsilon)\tilde{u}(\varepsilon)) \geq G( \tilde{u}(\varepsilon)) +\varepsilon t \|\tilde{u}(\varepsilon)\|
  \end{align*}
  for all $\varepsilon$ small enough such that $\|\tilde{u}(\varepsilon)\| \geq a$. For these $\varepsilon$ we have
  \begin{align*}
    &G(u^{n+1}) + \langle u^{n+1}, K^T \nabla F(g^{n+1})\rangle\\
    &\geq  G( \tilde{u}(\varepsilon)) + \langle \tilde{u}(\varepsilon), K^T \nabla F(g^{n+1})\rangle+ \varepsilon( t- c\|K\|) \|\tilde{u}(\varepsilon)\|\\
    &>  G( \tilde{u}(\varepsilon)) + \langle \tilde{u}(\varepsilon), K^T \nabla F(g^{n+1})\rangle.
  \end{align*}
  Additionally, one finds that
  \begin{align*}
    \|u^{n+1} - u^n\|^2 &= 
    \|(1+\varepsilon)\tilde{u}(\varepsilon) - u^n\|^2  \\&= \|\tilde{u}(\varepsilon) - u^n\|^2  + \varepsilon^2 \|\tilde{u}(\varepsilon)\|^2 + 2 \varepsilon \langle \tilde{u}(\varepsilon), \tilde{u}(\varepsilon) - u^{n}\rangle, 
  \end{align*}
  and observers that $\langle \tilde{u}(\varepsilon), \tilde{u}(\varepsilon) - u^{n}\rangle \geq 0$ at least as long as $\| \tilde{u}(\varepsilon)\| \geq \|u^n\|$, such that for $\varepsilon>0$ small enough we arrive at
  $$ E(u^{n+1}) > E(\tilde{u}(\varepsilon)),$$ which is a contradiction to $u^{n+1}$ being a minimizer of $E$. Therefore, $u^n$ remains bounded. Finally, notice that the boundedness of $u^n$ and $q^n$ together with the update $q^{n+1} = q^n + K\bar{u}^n - g^{n+1}$ immediately implies the boundedness of $g^n$. 
\end{proof}

\figBoundedness

The assumptions of Proposition \ref{prop:boundedness} are met for instance for image denoising with an $\ell^2$ fidelity term and smooth approximations to popular regularizations like truncated quadratic or TV$^q$. We will show some of these cases in our numerical results Section~\ref{sec:experiments}. Basically the condition ensures the coercivity of the problem which of course seems to be a reasonable criterion to obtain boundedness of the iterates. To get a better understanding of what the assumptions of Proposition \ref{prop:boundedness} mean we illustrate the idea in a 1D case in Fig.~\ref{fig:boundednessAssumption}.

Unfortunately, it is not clear that the accumulation points of the bounded sequence generated by Algorithm \eqref{eq:firstFormulation} are critical. For the convergence to a critical point we need the distance between successive iterates to vanish in the limit. The following result is similar to the observations by Zhang and Esser in \cite{Esser-Zhang-14}.

\begin{prop}
  If the sequences $(u^n)$, $(g^n)$ and $(q^n)$ remain bounded (and thus are convergent along subsequences), and we additionally have
that $\|u^{n+1} - u^n \| \rightarrow 0$ and $\|q^{n+1} - q^n\| \rightarrow 0$, then the iteration converges to critical points along subsequences. 
  \label{lem:distances}
\end{prop}

\begin{proof}
  If we pick a convergent subsequence (denoted with indices $m_n$) that converges to a point $(\hat{u}, \hat{g}, \hat{q})$, we have 
  \begin{align*}
    &\|q^{m_n-1} - q^{m_n}\|= \sigma \|g^{m_n} - K\bar{u}^{m_n-1}\|\\ &\geq \sigma \left(\|g^{m_n} - Ku^{m_n}\| - \|Ku^{m_n} - Ku^{m_n-1}\| - \theta \|Ku^{m_n-1} - Ku^{m_n-2}\|  \right)
  \end{align*}
  which ensures that $K\hat{u} = \hat{g}$. Additionally, taking the limit of the equation 
  \begin{align*}
    -\frac{1}{\tau}(u^{m_n} - u^{m_n-1}) \in K^Tq^{m_n} + \partial G(u^{m_n}),
  \end{align*}
  shows that the accumulation point satisfies 
  $$ -K^T \hat{q} \in  \partial G(\hat{u}).$$
  Using that $\hat{q} \in \partial F(\hat{g})$ (equation \eqref{eq:qIsSubgradient} along with the lower semicontinuity of $F$), as well as $K\hat{u} = \hat{g}$ yields
  $$ K^T \partial F(K\hat{u}) + \partial G(\hat{u}) \ni 0, $$
  and coincides with the definition of $\hat{u}$ being a critical point. 
\end{proof}
Note that if the assumptions of Proposition~\ref{prop:boundedness} are satisfied, the assumptions of Proposition~\ref{lem:distances} are also fulfilled.

\section{Numerical Experiments}

In this section we will provide different types of numerical examples. The first two subsections contain examples for which the overall function is strictly convex, but a splitting approach into a strongly convex and a nonconvex part clearly simplifies the numerical method, such that the primal-dual approach \eqref{eq:firstFormulation} can be applied with a convergence guarantee based on Theorem~\ref{thm:convergence}. Secondly, we consider the case of Mumford-Shah regularized denoising in the case where the overall energy is truly nonconvex. In this case we cannot guarantee the convergence a-priori, but can guarantee the boundedness of the iterates (Proposition~\ref{prop:boundedness}) and have a way to check a-posteriori if a critical point was found (based on Proposition~\ref{lem:distances}). Finally, we present numerical results on Mumford-Shah based inpainting as well as on image dithering to further illustrate the behavior of the algorithm in the fully nonconvex setting and to show that the numerical convergence often goes beyond the current theoretical guarantees.
\label{sec:experiments}

\subsection{Notation and Discretization}
Throughout the rest of this section $\Omega$ will denote a $d$-dimensional discretized rectangular domain. For images with $k$ channels $u : \Omega \rightarrow \R^k$ we denote the gradient discretization as $\nabla u : \Omega \rightarrow \R^{d \times k}$. We will make frequent use of the following norms for $g : \Omega \rightarrow \R^{d \times k}$:
$$
\norm{g}_{2,1} := \underset{x \in \Omega} \sum ~ \norm{g(x)}_2,
~~ \norm{g}_{2,2} := \sqrt{\underset{x \in \Omega} \sum ~ \norm{g(x)}_2^2},
$$
where $\norm{\cdot}_2$ denotes the Frobenius norm. In the following $\nabla$ is discretized by standard forward differences and we have the following bound on its norm $\norm{\nabla} \leq \sqrt{4 d}$ (see \cite{Chambolle-Pock-jmiv11}).

\subsection{Joint Denoising and Sharpening via Backwards Diffusion}
One of the most popular variational denoising methods is the Rudin-Osher-Fatemi (ROF) model \cite{Rudin-et-al-92} with the total variation (TV) as a regularizer. A possible interpretation in the vectorial case is given as
$$ \underset{u : \Omega \rightarrow \R^k} \min \enskip \frac{c}{2} \| u-f\|^2 + \| \nabla u \|_{2,1}. $$
Based on Lemma \ref{lem:equivalenceConvexityTypes} we can see that the energy to be minimized is $c$-strongly convex, which gives us the freedom to introduce further $\omega$-semiconvex terms and, as long as $c > \omega$, still be able to use the primal-dual splitting approach as an efficient minimization method. One possible example would be to incorporate a sharpening/edge enhancement term of the form $-\frac{\omega}{2}\|\nabla u \|_{2,2}^2$ with $\omega < c \|\nabla\|^{-2}$. The latter constraint is needed for the assumptions of Theorem \ref{thm:convergence} to hold and leads to \eqref{eq:denoiseAndSharpen} being a convex problem. Note that incorporating a term like $-\frac{\omega}{2}\|\nabla u \|_{2,2}^2$ into the energy can be interpreted as an implicit step for the backward heat equation. Thus, if a blur is assumed to follow a diffusion process, the term $-\frac{\omega}{2}\|\nabla u \|_{2,2}^2$ would aim at removing the blur. Similar approaches to image sharpening have been investigated in the literature, mostly in the context of inverse diffusion equations (cf.~\cite{gilboa02,gilboa04}). 

The full energy minimization approach for joint denoising and sharpening could be
\begin{align}
  \underset{u : \Omega \rightarrow \R^k} \min ~ \frac{c}{2} \| u-f\|^2 + \| \nabla u \|_{2,1} -\frac{\omega}{2}\|\nabla u \|_{2,2}^2 + \iota_{[0,1]}(u)
  \label{eq:denoiseAndSharpen}
\end{align}
where 
\begin{equation}
  \iota_{[0,1]}(u) = 
  \begin{cases}
    0, \qquad &0 \leq u(x) \leq 1, \enskip \forall x \in \Omega\\
    \infty, \qquad &\text{else,}
  \end{cases}
\end{equation}
restricts the range of $u$ to be between zero and one. Generally, we expect the total variation term to be dominant for small oscillations in the data, thus removing some of the noise, while the sharpening term will be dominant on large edges leading to an increased contrast. Notice that one can also iterate the above model by replacing $f$ with the previous iterate $u^k$ starting with $u^0=f$, which results in a combination of the TV flow \cite{Andreu00} and inverse diffusion. 

Notice that \eqref{eq:denoiseAndSharpen} is a good example for a functional which is convex but difficult to minimize without splittings that divide the energy into a strongly convex and a nonconvex part. A convex splitting would have to compute a proximity operator with respect to $ \frac{c}{2} \| u-f\|^2  -\frac{\omega}{2}\|\nabla u \|^2 + \iota_{[0,1]}(u)$, which is difficult to solve without an additional splitting. Forward-backward splittings have to use $\| \nabla u \|_{2,1} $ in the backward step, since it is nondifferentiable, such that one has to solve a full TV-denoising problem at each iteration. Similar problems occur in differences of convex approaches. 

To the best knowledge of the authors this manuscript is the first work to theoretically justify the splitting into a convex and a nonconvex part by considering the problem 
\begin{align}
 \underset{u, g} \min ~ \frac{c}{2} \| u-f\|^2 + \| g \|_{2,1} -\frac{\omega}{2}\|g \|_{2,2}^2 + \iota_{[0,1]}(u) \quad \text{subject to} \quad g = \nabla u.
  \label{eq:denoiseAndSharpenSplit}
\end{align}
Fig.~\ref{fig:denoiseAndSharpen} shows the result of the flow arising from \eqref{eq:denoiseAndSharpen} (called \textit{enhanced TV flow} in Fig.~\ref{fig:denoiseAndSharpen}) in comparison to a plain TV flow for $c = 30$ and $\omega = 0.7 \frac{c}{\|\nabla\|^2}$. While both, the TV and the inverse diffusion TV flow, are able to remove the noise from the image, we can see that in the course of the iteration, the inverse diffusion TV flow results remain significantly sharper. Strong image edges even get enhanced during the iterations, such that due to the nature of backward diffusion, the iterations have to be stopped at some point to avoid significant overshooting. 

We do not claim the above energy to be a state of the art image enhancement technique. However, the above experiment shows that we can add a certain amount of a semiconvexity to any strongly convex energy without limiting our abilities to tackle the resulting problem with an efficient splitting approach. We believe that the inclusion of such nonconvexities in the regularization might aid in the design of energy functionals that describe the desired effects even more accurately than purely convex ones. 

\figDenoiseAndSharpen

\subsection{Illumination Correction}
In the previous subsection, the linear operator in the nonconvex term was exactly the same as the one we would typically use for the splitting in solving a plain TV problem. Additionally, the proximity operator for the TV plus the nonconvex could be solved in closed form. In this subsection, we would like to use a toy example to show that this does not have to be the case and one could imagine adding nonconvex terms which are treated completely independent of the convex ones. Consider for instance terms of the form 
\begin{align}
  N(Au) = \|(Au-r)^2 - e \|_1,
  \label{eq:nonconvexIlluTerm}
\end{align}
where $A$ is a linear operator, $r$ and $e$ are given values, and the square is a componentwise operation. Due to the $\ell^1$ norm, this term has the interpretation that $(Au-r)^2 - e $ is sparse, i.e. that many components meet $((Au-r)_{i,j})^2 = e_{i,j}$. Looking at the function $N(Au)$ in 1D in Fig.~\ref{fig:n_of_u_example} it is obvious that such an $N$ is nonconvex and nondifferentiable for positive $e$. It is, however, $\omega$-semiconvex, such that adding $N(Au)$ with a sufficiently small weight to any strongly convex function leads to an overall convex energy and our convergence theory regarding a splitting into a convex and nonconvex part applies. 

\figShapeOfN

As a toy problem, consider again the ROF denoising model for three color channels ($k=3$), with the additional idea to perform a particular illumination correction, by requiring many gray values defined as red channel plus green channel plus blue channel over three, to be close either a number $b_1$ or a number $b_2$. In other words, we introduce the term $\|(\frac{1}{3}\sum_k u_{i,j,k} - r)^2 - e\|_1$, where $r$ and $e$ are just numbers in $\mathbb{R}$. For instance, for $r=0.6$ and $e=(0.3)^2$ the term would prefer many intensity values to be either $0.3$ or $0.9$. With this term we can modify the tonemapping; notice that intensity contrast enhancement is a special case obtained for $r = 0.5$ and $e>(0.5)^2$, where the relevant part of $N$ is mere the negative quadratic of the intensity. 

The energy of our variational model reads as follows:
\begin{align}
  \underset{u : \Omega \rightarrow \R^3} \min ~ \frac{c}{2} \| u-f\|^2 + \| \nabla u \|_{2,1} + \frac{\omega}{2} N(Au)+ \iota_{[0,1]}(u).
  \label{eq:denoiseAndIllucorrect}
\end{align}
Writing the color image $u$ in vector form, we can introduce a new variable $g$ along with the constraint $Ku = g$, where 
\begin{align}
  K = \left (
    \begin{array}{ccc}
      \nabla & 0 & 0 \\
      0 & \nabla & 0 \\
      0 & 0 & \nabla \\
      & A &
    \end{array}
  \right ) ,
\end{align}
and $\nabla$ denotes the discrete gradient operator of a single channel image in matrix form. We can apply the proposed algorithm which then decouples the minimization of the nonconvex term involving $N$ from the rest of the energy and all proximity that need to be evaluated in the algorithm have a closed from solution. 

\figIllucorrectionOne

Fig.~\ref{fig:denoiseAndIllucorrect} shows some results of the method obtained by minimizing \eqref{eq:denoiseAndIllucorrect}. We compare the original noisy image with the one obtained by \eqref{eq:denoiseAndIllucorrect} without the nonconvex term (i.e. $\omega = 0$), and with the nonconvex term (i.e. $\omega = \alpha - \varepsilon$ for a small $\varepsilon$) for $r=0.6$, $e=(0.3)^2$. It is interesting to see how strong the effect of the nonconvex term is: We see big intensity changes and particularly a stronger denoising effect in the background. To illustrate that our idea of using \eqref{eq:nonconvexIlluTerm} to obtain many entries that meet $((Au-r)_{i,j})^2 = e_{i,j}$, we plotted the sorted pixel intensities of the plain TV denoising solution as well as the  sorted pixel intensities of our model with the additional nonconvex term. As we can see there is a drastic increase in the number of pixel for which the intensity is exactly  $0.3$ or $0.9$.

Again, the exaggerated effect of Fig.~\ref{fig:denoiseAndIllucorrect} was for illustration purposes. One could imagine cases where the original image has a low dynamic range and does not fully use the possible range of values from zero to one. In this case the approach seems more reasonable and we obtain the results shown in Fig.~\ref{fig:denoiseAndIllucorrect2}. 

\figIllucorrectionTwo

While we are not certain if the idea of incorporating the nonconvex term for illumination correction has wide applicability, our experiments do show that whenever one has a strongly convex energy, one has the freedom to model any desirable image property with a semiconvex term without complicating the minimization algorithm. As we have seen, the condition of being semiconvex is rather weak (particularly when working in a bounded domain) such that one has a great freedom in the design of such terms. Furthermore, we have seen that the effect of these additional nonconvex terms is by no means negligible. Additionally we would like to mention, that since the effect of the nonconvex terms is influenced by their weight and therefore by the strong convexity constant, the effect of any nonconvex term can be enhanced in an iterative fashion by computing the corresponding flow (which for a simple $\ell^2$ data term amounts to replacing the data $f$ by the previous iterate $u^k$).

\subsection{Mumford-Shah Regularization}
In the remaining subsections we will give examples where the overall energy is \emph{nonconvex}. Our main convergence Theorem does not hold, but we observe convergence experimentally nonetheless. 

The Mumford-Shah functional provides means to compute a discontinuity preserving piecewise smooth approximation of the input image. Strekalovskiy and Cremers \cite{MS-tr} initially proposed the nonconvex primal-dual algorithm \eqref{eq:firstFormulation} to efficiently find a minimizer of the Mumford-Shah problem. The following discretization was considered:
\begin{equation}
  \underset{u : \Omega \rightarrow \R^{k}} \min \enskip \underset{x \in \Omega} \sum \enskip (u(x) - f(x))^2 + (\widehat R_{MS} \circ \norm{\cdot}_2)(\nabla u(x))
  \label{eq:MS-nonsmooth}
\end{equation}
where $\widehat R_{MS} : \R \rightarrow \R$ is the truncated quadratic regularizer
\begin{equation}
  \widehat R_{MS}(t) = \min \{ \lambda, \, \alpha t^2 \}.
  \label{eq:truncQuad}
\end{equation}

Since \eqref{eq:truncQuad} is not $\omega$-semiconvex due to the truncation, it seems to be necessary for the step size $\sigma$ to approach infinity in order for the algorithm to converge. The authors of \cite{MS-tr} employed the variable step size scheme from \cite[Algorithm 2]{Chambolle-Pock-jmiv11} which along with $\sigma \rightarrow \infty$ has the nice property that it provably converges in the fully convex setting with rate $\mathcal{O}(1/n^2)$ if the dataterm is strongly convex. 

We propose to use a slight modification of \eqref{eq:truncQuad} in order to obtain $\omega$-semiconvexity of the regularizer. We consider the smoothed version $\widehat R_{MS}^{\varepsilon}:\R\rightarrow\R$ of \eqref{eq:truncQuad} as described in \cite{fornasier}. It is shown in Fig.~\ref{fig:truncQuad} and given as
\begin{equation}
  \widehat R_{MS}^{\varepsilon}(t) = 
  \begin{cases}
    \alpha \, t^2, &\quad t < s_1\\
    \pi(t), &\quad s_1 \leq t \leq s_2\\
    \lambda, &\quad t > s_2,
  \end{cases}
  \label{eq:truncQuadSmooth}
\end{equation}
where $s_1 = \sqrt{\lambda / \alpha} - \varepsilon$, $s_2 = \sqrt{\lambda / \alpha} + \varepsilon$. $\pi(t) = A (t - s_2)^3 + B (t - s_2)^2 + C$ is a cubic spline from \cite{fornasier} with constants
\begin{equation}
  \begin{aligned}
    &A = -\frac{\alpha}{4 \varepsilon},\enskip
    B = -\frac{\alpha (2 \sqrt{\lambda / \alpha} + \varepsilon)}{4\varepsilon},\enskip
    C = \lambda.
  \end{aligned}
\end{equation}

\begin{lemma}
  The smoothed truncated quadratic regularizer \eqref{eq:truncQuadSmooth} is $\omega$-semiconvex with constant 
  \begin{equation}
    \omega = \frac{\alpha \, (2 + \varepsilon_0)}{2 \, \varepsilon_0}
  \end{equation}
  where $\varepsilon = \varepsilon_0 \sqrt{\lambda / \alpha}$. Furthermore, the composition $R_{MS}^{\varepsilon} : \R^{d \times k} \rightarrow \R$ given by
  \begin{equation}
    R_{MS}^{\varepsilon}(g) = (\widehat R_{MS}^{\varepsilon} \circ \norm{\cdot}_2)(g)
  \end{equation}
  is also $\omega$-semiconvex with the same constant.
\end{lemma}
\begin{proof}
  A sufficient condition for $\widehat R_{MS}^{\varepsilon} + \frac{\omega}{2}t^2$ being convex is that the second derivative is nonnegative:
  \begin{equation}
    (\widehat R_{MS}^{\varepsilon})''(t) + \omega \geq 0, \quad \forall t \in \R.
  \end{equation}
  We only have to consider the region $s_1 \leq t \leq s_2$, since $\widehat R_{MS}^{\varepsilon}$ is nonconvex only in that part:
  \begin{equation}
    \begin{aligned}
      &\underset{s_1 \leq t \leq s_2} \min \, \{ 6 A (t - s_2) + 2B \} + \omega \geq 0 \\
      &\Rightarrow \quad \omega \geq  -2B = \frac{\alpha \, (2 + \varepsilon_0)}{2 \, \varepsilon_0}.
    \end{aligned}
  \end{equation}
  Since $\widehat R_{MS}^{\varepsilon}(t) + \frac{\omega}{2}t^2$ is increasing and convex, the composition $ ((\widehat R_{MS}^{\varepsilon} + \frac{\omega}{2} (\cdot)^2) \circ \norm{\cdot}_2)(g)$ is convex as well.
\end{proof}

\figTruncQuad 
\paragraph{Evaluation of the Proximal Mapping.}
Next, we describe how to evaluate the proximal mapping of the truncated quadratic regularizer $R^{\varepsilon}_{MS}$. For that we note that a slight modification of \cite[Theorem 4]{DecomposeProx} to our setting of $\omega$-semiconvex functions also holds. This allows us to reduce the evaluation of the proximal mapping of $\omega$-semiconvex regularizers to single variable optimization problems.
\begin{rem}
  Let $h : \R \rightarrow \R$ be increasing and $\omega$-semiconvex. Furthermore, let $f := h \circ \norm{\cdot}$ and $\tau^{-1} > \omega$. 
  Then the proximal mapping for $f$ is given as
  \begin{equation}
    (I + \tau \partial f)^{-1}(\tilde u) = (I + \tau \partial h)^{-1}(\norm{\tilde u}) \frac{\tilde u}{\norm{\tilde u}}
  \end{equation}
  with the convention $0/0=0$.
\end{rem}

The proximal mapping $(I + \tau \partial \widehat R_{MS}^{\varepsilon})^{-1}$ can be evaluated by considering the three piecewise parts of $\widehat R_{MS}^{\varepsilon}$ individually.

\subsubsection{Piecewise Smooth Approximations}
\figDama
First we consider a smoothed version of the functional \eqref{eq:MS-nonsmooth} to validate experimentally if the algorithm also converges in the truly nonconvex setting. 

The energy we are aiming to minimzie is given as
\begin{equation}
  \underset{u : \Omega \rightarrow \R^{k}} \min \enskip \underset{x \in \Omega} \sum ~ R_{MS}^{\varepsilon}(\nabla u(x)) + \underset{x \in \Omega} \sum ~ (f(x) - u(x))^2.
  \label{eq:MSDenoising}
\end{equation}
In Fig.~\ref{fig:dama} we show a piecewise smooth approximation of the input image computed by minimizing the functional \eqref{eq:MSDenoising} with the nonconvex primal-dual algorithm. The chosen parameters were $\alpha = 10$, $\lambda = 0.1$ and $\varepsilon_0 = 0.5$. Unlike the approach \cite{MS-tr} we use completely constant step sizes $\sigma = 2 \omega$, $\tau \sigma \leq \norm{\nabla}^{-2}$. We see that the end result is dependent on the initialization. Using a smooth initialization $u^0=0$ yields an overall smoother image than using the input image $u^0=f$. 

Since we do not have a full convergence theorem for energies which are completely nonconvex, we validate the convergence for this and the following two examples numerically. The convergence results are detailed in Appendix~\ref{lbl:appendixConv}.

\subsubsection{Image Inpainting}
\figCracktip
Image inpainting aims to fill in an inpainting region $I \subset \Omega$ in the image. We aim to minimize the following functional
\begin{equation}
  \underset{u : \Omega \rightarrow \R^{k}} \min \enskip \underset{x \in \Omega} \sum R_{MS}^{\varepsilon}(\nabla u(x)) + \underset{x \in \Omega \setminus I} \sum \chi(f(x), u(x)),
  \label{eq:MSInpainting}
\end{equation}
where $\chi$ is the following inpainting indicator function
\begin{equation}
  \chi(f, u) = \begin{cases}
    0, &\quad f = u\\
    \infty, &\quad f \neq u.
  \end{cases}
\end{equation}

The so-called cracktip problem \cite{Bonnet-01} is an instance of Mumford-Shah regularized image inpainting where an analytical globally optimal solution is known.

In Fig.~\ref{fig:cracktip} we show the results of the nonconvex primal-dual algorithm applied to the inpainting problem. The image size here is $127 \times 127$ and the according parameters are $\alpha \approx 96.82$ and $\lambda = 0.5$. We chose $\varepsilon_0 = 0.9$. We found that by using constant step sizes from the beginning, the algorithm often gets stuck in bad local minima. There are several methods to avoid getting stuck in bad local minima, such as graduated nonconvexity. Another method that works very well in practice is using large step sizes in the beginning and decreasing the step size according to some parameter $\gamma$. We employed the adaptive step size scheme
\begin{equation}
  \theta_{n}=1 / \sqrt{1 + 2 \gamma \tau_n}, \sigma_{n+1}=\sigma_n / \theta_n, \tau_{n+1}=\tau_n \theta_n,
\end{equation}
and as soon as $\sigma_n$ reached $2 \omega$, we kept the step sizes fixed and set $\theta = 1$. It was shown experimentally in \cite{MS-tr} that this step size scheme works remarkably well in practice. The step size parameters used are given in Fig.~\ref{fig:cracktip}, and we set $\tau_0 = \norm{\nabla}^2 / \sigma_0$. Interestingly, for the right choice of parameters, the nonconvex PDHG algorithm converges to the globally optimal solution.

\subsection{Image Dithering}
\figDithering
In this section we propose a simple nonconvex variational model for image dithering. Given a grayscale input image $f : \Omega \rightarrow \R$, dithering aims to produce a \emph{binary} image that is a visually similar continuous approximation to the given image $f$. We model visual similarity by a convolution with a gaussian kernel $k$ of standard deviation $\sigma = 1.75$.

The variational approach is a deconvolution problem with the constraint that the solution $u$ is binary. In order to enforce that constraint, we add a concave regularization term on $u$. The overall cost function is:
\begin{equation}
  \label{eq:dithering}
  \underset{u : \Omega \rightarrow \R} 
  \min ~ \underbrace{\norm{k * u - f}^2}_{G(u)} + \underbrace{\lambda \, \underset{x \in \Omega} \sum \, -(2 u(x) - 1)^2 + \iota_{[0,1]}(u)}_{F(u)}.
\end{equation}
Note that both proximal mappings in $F$ and $G$ are simple to evaluate. Furthermore, it can be verified quickly that $F$ is $\omega$-semiconvex with $\omega = 8 \lambda$. The proximal mapping in $G$ has an efficient solution by means of the FFT (cf. \cite{Chambolle-Pock-jmiv11}) and the proximal mapping in $F$ has the following analytic solution given pointwise at $x \in \Omega$:
$$
(I+\tau \partial \widehat F)^{-1}(\tilde g(x)) = \text{proj}_{[0,1]}\left( \frac{\tilde g(x) - 4 \lambda \tau}{1 - 8 \lambda \tau} \right ).
$$
Minimizing \eqref{eq:dithering} with a Forward-backward splitting approach would require to move the indicator function $\iota_{[0,1]}$ into $G$ to make $F$ differentiable. The proximal mapping in $G$ would then be difficult to evaluate, as an FFT based solver is not applicable anymore.

In Fig.~\ref{fig:dithering} we show the result of minimizing \eqref{eq:dithering} directly with the PDHG algorithm. The parameters used in the example were $\lambda=0.01$, $\sigma = 2 \omega$, $\tau = 1 / \sigma$. The initial $u^0$ was chosen as a thresholded random image and we set $q^0=0$. By comparing the final result with a thresholded version of the result we can see that it is almost binary. Overall the algorithm required around 200 iterations to converge.





\section{Conclusion}
\label{sec:conclusions}
In this paper we studied the PDHG method for minimizing energies of the form $E(u) = G(u) + F(Ku)$ for $\omega$-semiconvex functions $F$. We analyzed its convergence and were able to prove that splittings into a strongly convex and a nonconvex part converge if the strong convexity dominates the nonconvexity and the algorithm parameters are chosen appropriately. We constructed a simple example illustrating that the choice of algorithm parameters arising from the convergence proof are not only sufficient but necessary for the convergence. In the more general case where the total energy is nonconvex we were able to proof boundedness of the iterates and found an a-posteriori criterion to verify the convergence of the algorithm to a critical point. 

In terms of the practical relevance of our convergence proof, we demonstrated in two numerical examples that any variational model with a strongly convex energy gives us the freedom to introduce semiconvex terms that encourage a certain behavior which is impossible to represent in a framework with merely convex terms. Thanks to our analysis of the PDHG method the minimization of such energies with additional nonconvex terms does not increase the computational costs as long as the nonconvex term has an easy-to-evaluate proximity operator. 

Additionally, we demonstrated that the convergence of the investigated PDHG method goes beyond the theoretical guarantees: For Mumford-Shah denoising, Mumford-Shah inpainting as well as for image dithering we observed the convergence of the algorithm for all our numerical experiments and could use the a-posteriori criterion to verify that the final result is a critical point of the original problem. 

\bibliographystyle{splncs}
\bibliography{references,allrefs,myrefs}

\begin{appendix}

\section{Proofs of Theorems}
\label{lbl:appendixProof}
In the appendix, we prove our main convergence result (Theorem \ref{thm:convergence}). We start our convergence analysis by deriving an equation that gives us some insight about how certain iteration errors behave. It is formulated in a Lemma and derived in a rather general form such that it can be a basis for further/future convergence analysis. 
\begin{lemma}
  \label{lem:estimate}
  Let $(v^n, \rho^n, \xi^n)$ be any sequence such that the following holds
  \begin{align}
    \label{eq:firstOpti}
    0 &= \sigma(\rho^{n+1} - K\bar{v}^n) + \partial F(\rho^{n+1}) - \xi^n,  \\
    \label{eq:update}
    \xi^{n+1} &= \xi^n + \sigma(K\bar{v}^n - \rho^{n+1}),\\
    \label{eq:secondOpti}
    0 &= \frac{1}{\tau}(v^{n+1} - v^n) + K^T\xi^{n+1} + \partial G(v^{n+1}). 
  \end{align}
  Let $(u^n, g^n, q^n)$ be the iterates generated by Algorithm \eqref{eq:firstFormulation}, and let us denote $u_e^n = u^n - v^n$, $g_e^n = g^n - \rho^n$, $q_e^n = q^n - \xi^n$. Then the following estimate holds:
  \begin{align*}
    0 =& \frac{1}{2\tau}(\|u_e^{n+1}\|^2 - \|u_e^n\|^2)  + \frac{1}{2\sigma} (\|q_e^{n+1}\|^2 - \|q_e^n\|^2) \\
    &+ \langle q_e^{n+1}, Ku_e^{n+1} \rangle - \langle q_e^{n}, K\bar{u}_e^{n} \rangle   + \frac{1}{2\tau}\| u_e^{n+1} - u_e^n \|^2 \\
    & - \frac{\sigma}{2} \theta \left(\|Ku_e^{n}\|^2 - \|Ku_e^{n-1}\|^2 \right) - \frac{\sigma(\theta + \theta^2)}{2} \|Ku_e^n - Ku_e^{n-1}\|^2\\
    &+ \langle \partial G(u^{n+1})-\partial G(v^{n+1}), u^{n+1} - v^{n+1} \rangle-\frac{\sigma}{2}\|Ku_e^{n}\|^2   \\
    &+ \langle \partial F(g^{n+1}) - \partial F(\rho^{n+1}), g^{n+1} - \rho^{n+1}\rangle +\frac{\sigma}{2}\|g_e^{n+1}\|^2.
  \end{align*}
\end{lemma}
\begin{rem}
  We slightly abused the notation in the above equations by using the
  notation of sets like $\partial F(g^{n})$ to denote an element from
  this set. Here, this element is uniquely determinted by the algorithm
  since the subproblems are strictly convex. This has the advantage of
  avoiding even more variables and different names and has been done for
  the sake of clarity.
\end{rem}
\begin{proof}[Proof of Lemma \ref{lem:estimate}]
  Notice that the conditions \eqref{eq:firstOpti}, \eqref{eq:update}, \eqref{eq:secondOpti}, are exactly the optimality conditions to the iteration constructed by Algorithm \eqref{eq:firstFormulation}. We will later be able to either choose $(v^n, \rho^n, \xi^n)=(u^{m}, g^{m}, q^{m})$ or $(v^n, \rho^n, \xi^n)=(\hat{u}, \hat{g}, \hat{q})$ with $(\hat{u}, \hat{g}, \hat{q})$ being a critical point to our original optimization problem \eqref{eq:optiProb}.

  We subtract \eqref{eq:firstOpti}, \eqref{eq:update}, \eqref{eq:secondOpti}, from the optimality conditions for $(u^n, g^n, q^n)$ generated by Algorithm \eqref{eq:firstFormulation} to obtain:
  \begin{align*}
    0 &= \sigma(g_e^{n+1} - K\bar{u}_e^{n}) + (\partial F(g^{n+1}) - \partial F(\rho^{n+1})) - q_e^n,  \\
    0 &= \frac{1}{\tau}(u_e^{n+1} - u_e^n) + K^Tq_e^{n+1} + (\partial G(u^{n+1})-\partial G(v^{n+1})), \\
    q_e^{n+1} &= q_e^n + \sigma(K\bar{u}_e^n - g_e^{n+1}).
  \end{align*}
  Taking the inner product of the first equation with $g_e^{n+1}$ yields
  \begin{align*}
    0 =& \sigma(\|g_e^{n+1}\|^2 - \langle K\bar{u}_e^{n}, g_e^{n+1} \rangle) - \langle q_e^n, g_e^{n+1} \rangle  \\
    &+ \langle \partial F(g^{n+1}) - \partial F(\rho^{n+1}), g^{n+1} - \rho^{n+1}\rangle.  \\
    \Rightarrow 0 = & \frac{\sigma}{2}(\|g_e^{n+1}\|^2 - \|K\bar{u}_e^{n}\|^2 + \| g_e^{n+1} - K\bar{u}_e^{n}\|^2)  -  \langle q_e^n, g_e^{n+1} \rangle  \\
    &+ \langle \partial F(g^{n+1}) - \partial F(\rho^{n+1}), g^{n+1} - \rho^{n+1}\rangle .
  \end{align*}
  Taking the inner product of the second equation with $u_e^{n+1}$ yields
  \begin{align*}
    0 =& \frac{1}{\tau}(\|u_e^{n+1}\|^2 - \langle u_e^n, u_e^{n+1} \rangle ) + \langle q_e^{n+1}, Ku_e^{n+1} \rangle \\
    &+ \langle \partial G(u^{n+1})-\partial G(v^{n+1}), u^{n+1} - v^{n+1} \rangle. \\
    \Rightarrow
    0 =& \frac{1}{2\tau}(\|u_e^{n+1}\|^2 - \|u_e^n\|^2 + \| u_e^{n+1} - u_e^n \|^2 )+ \langle q_e^{n+1}, Ku_e^{n+1} \rangle \\
    &+ \langle \partial G(u^{n+1})-\partial G(v^{n+1}), u^{n+1} - v^{n+1} \rangle. \\
  \end{align*}
  Adding the two equations we obtain
  \begin{align*}
    0 =& \frac{1}{2\tau}(\|u_e^{n+1}\|^2 - \|u_e^n\|^2 + \| u_e^{n+1} - u_e^n \|^2 ) \\
    &+ \langle q_e^{n+1}, Ku_e^{n+1} \rangle - \langle q_e^{n}, K\bar{u}_e^{n} \rangle  +  \langle q_e^n, K\bar{u}^n - g_e^{n+1} \rangle \\
    &+\frac{\sigma}{2}(\|g_e^{n+1}\|^2 - \|K\bar{u}_e^{n}\|^2 + \| g_e^{n+1} - K\bar{u}_e^{n}\|^2)  \\
    &+ \langle \partial G(u^{n+1})-\partial G(v^{n+1}), u^{n+1} - v^{n+1} \rangle \\
    &+ \langle \partial F(g^{n+1}) - \partial F(\rho^{n+1}), g^{n+1} - \rho^{n+1}\rangle .
  \end{align*}
  Observe that
  \begin{align*}
    \|q_e^{n+1}\|^2 =& \|q_e^n + \sigma(K\bar{u}_e^n - g_e^{n+1})\|^2 \\
    =& \|q_e^n\|^2 + \sigma^2 \|K\bar{u}_e^n - g_e^{n+1}\|^2  + 2\sigma \langle q_e^n,K\bar{u}_e^n - g_e^{n+1} \rangle , \\
    \Rightarrow \langle q_e^n,K\bar{u}_e^n - g_e^{n+1} \rangle  =& \frac{1}{2\sigma} (\|q_e^{n+1}\|^2 - \|q_e^n\|^2) - \frac{\sigma}{2} \|K\bar{u}_e^n - g_e^{n+1}\|^2.
  \end{align*}
  Plugging this into the previous equation yields
  \begin{align*}
    0 =& \frac{1}{2\tau}(\|u_e^{n+1}\|^2 - \|u_e^n\|^2 + \| u_e^{n+1} - u_e^n \|^2 ) \\
    &+ \langle q_e^{n+1}, Ku_e^{n+1} \rangle - \langle q_e^{n}, K\bar{u}_e^{n} \rangle  + \frac{1}{2\sigma} (\|q_e^{n+1}\|^2 - \|q_e^n\|^2)  \\
    &+\frac{\sigma}{2}(\|g_e^{n+1}\|^2 - \|K\bar{u}_e^{n}\|^2 )  \\
    &+ \langle \partial G(u^{n+1})-\partial G(v^{n+1}), u^{n+1} - v^{n+1} \rangle \\
    &+ \langle \partial F(g^{n+1}) - \partial F(\rho^{n+1}), g^{n+1} - \rho^{n+1}\rangle.
  \end{align*}
  Notice that due to 
  $$ \bar{u}^n =  u^n + \theta (u^n - u^{n-1}),$$
  we have
  \begin{align*}
    \|K\bar{u}_e^n \|^2 =& \|Ku_e^n\|^2 + \theta^2 \|Ku_e^n - Ku_e^{n-1}\|^2 + 2 \theta \langle Ku_e^n, Ku_e^n -Ku_e^{n-1} \rangle \\
    =& \|Ku_e^n\|^2 + \theta^2 \|Ku_e^n - Ku_e^{n-1}\|^2 \\
    &+ \theta\left( 2\|Ku_e^n\|^2 -  2\langle Ku_e^n, Ku_e^{n-1} \rangle + \|Ku_e^{n-1}\|^2 - \|Ku_e^{n-1}\|^2 \right) \\
    =& \|Ku_e^n\|^2 +\theta^2 \|Ku_e^n - Ku_e^{n-1}\|^2 \\
    & + \theta\left(\|Ku_e^{n}\|^2 - \|Ku_e^{n-1}\|^2 + \|Ku_e^n - Ku_e^{n-1}\|^2\right)\\
    =& \|Ku_e^n\|^2 + \theta (\|Ku_e^{n}\|^2 - \|Ku_e^{n-1}\|^2) + (\theta + \theta^2) \|Ku_e^n - Ku_e^{n-1}\|^2.
  \end{align*}
  Using the above equation in our main estimate, we obtain
  \begin{align*}
    0 =& \frac{1}{2\tau}(\|u_e^{n+1}\|^2 - \|u_e^n\|^2)  + \frac{1}{2\sigma} (\|q_e^{n+1}\|^2 - \|q_e^n\|^2) \\
    &+ \langle q_e^{n+1}, Ku_e^{n+1} \rangle - \langle q_e^{n}, K\bar{u}_e^{n} \rangle   + \frac{1}{2\tau}\| u_e^{n+1} - u_e^n \|^2 \\
    & - \frac{\sigma}{2} \theta \left(\|Ku_e^{n}\|^2 - \|Ku_e^{n-1}\|^2 \right) - \frac{\sigma(\theta + \theta^2)}{2} \|Ku_e^n - Ku_e^{n-1}\|^2\\
    &+ \langle \partial G(u^{n+1})-\partial G(v^{n+1}), u^{n+1} - v^{n+1} \rangle-\frac{\sigma}{2}\|Ku_e^{n}\|^2   \\
    &+ \langle \partial F(g^{n+1}) - \partial F(\rho^{n+1}), g^{n+1} - \rho^{n+1}\rangle +\frac{\sigma}{2}\|g_e^{n+1}\|^2,
  \end{align*}
  and hence the assertion. 
\end{proof}
\begin{lemma}
  \label{lem:estimate2}
  Using the notation of Lemma \ref{lem:estimate}, let $(v^n, \rho^n, \xi^n)=(\hat{u}, \hat{g}, \hat{q})$ not depend on $n$. Then the following estimate holds:
  \begin{align*}
    0 \geq & \frac{1}{2\tau}(\|u_e^{n+1}\|^2 - \|u_e^n\|^2)  + \frac{1}{2\sigma} (\|q_e^{n+1}\|^2 - \|q_e^n\|^2) \\
    &+ \frac{1}{2\tau}\| u_e^{n+1} - u_e^n \|^2 - \frac{\sigma \theta^2}{2} \|Ku_e^n - Ku_e^{n-1}\|^2 \\
    &+ \frac{\theta}{2\tau}\|u^n - u^{n-1}\|^2- \frac{\sigma \theta }{2} \|Ku_e^n - Ku_e^{n-1}\|^2 \\
    &+ \theta \left( G(u^n) - G(u^{n-1}) -  \langle \bar{q}, K(u^{n-1}-u^n) \rangle \right) \\
    &+ \langle q_e^{n+1}, Ku_e^{n+1} \rangle - \langle q_e^{n}, Ku_e^{n} \rangle \\
    & - \frac{\sigma}{2} \theta \left(\|Ku_e^{n}\|^2 - \|Ku_e^{n-1}\|^2 \right)\\
    &+ \langle \partial G(u^{n+1})-\partial G(\hat{u}), u^{n+1} - \hat{u} \rangle-\frac{\sigma}{2}\|Ku_e^{n}\|^2   \\
    &+ \langle \partial F(g^{n+1}) - \partial F(\hat{g}), g^{n+1} -  \hat{g}\rangle +\frac{\sigma}{2}\|g_e^{n+1}\|^2.
  \end{align*}
\end{lemma}
\begin{proof}
  We start with the equality from Lemma \ref{lem:estimate} and observe that now
  \begin{align*}
    - \langle q_e^{n}, K\bar{u}_e^{n} \rangle =&  - \langle q_e^{n}, Ku_e^{n} \rangle + \theta \langle q^{n} - \hat{q}, K(u^{n-1}-u^n) \rangle \\
    =&  - \langle q_e^{n}, Ku_e^{n} \rangle + \theta \langle q^{n}, K(u^{n-1}-u^n) \rangle  - 
    \theta \langle \hat{q}, K(u^{n-1}-u^n) \rangle. 
  \end{align*}
  Using the fact that $u^{n}$ is determined by solving a minimization problem, it is obvious that the corresponding energy evaluated at the minimizer $u^{n}$ is lower than the energy evaluated at the previous iterate, $u^{n-1}$. Hence, 
  \begin{align*}
    & \frac{1}{2\tau}\|u^n - u^{n-1}\|^2 + \langle u^n, K^T q^{n} \rangle + G(u^n) \leq \langle u^{n-1}, K^T q^{n} \rangle + G(u^{n-1}), \\
    \Rightarrow & \langle u^{n-1} -  u^n, K^T q^{n} \rangle  \geq \frac{1}{2\tau}\|u^n - u^{n-1}\|^2 + G(u^n) - G(u^{n-1}),
  \end{align*}
  such that
  \begin{align*}
    - \langle q_e^{n}, K\bar{u}_e^{n} \rangle \geq&  - \langle q_e^{n}, Ku_e^{n} \rangle \\
    &+ \theta \left(\frac{1}{2\tau}\|u^n - u^{n-1}\|^2 + G(u^n) - G(u^{n-1}) -  \langle q, K(u^{n-1}-u^n) \rangle \right).
  \end{align*}
  Using this estimate in the equality from Lemma \ref{lem:estimate} as well as reordering the terms leads to the assertion. 
\end{proof}

\begin{rem}
  It is easy to see that a point such that we can choose $(v^n, \rho^n, \xi^n)=(\hat{u}, \hat{g}, \hat{q})$ is exactly a critical point to our optimization problem \eqref{eq:optiProb}. Hence, such a point exists whenever the energy has at least one critical point. 
\end{rem}

We will now use the lengthy estimate of Lemma \ref{lem:estimate2} to prove our convergence result with the idea that each line of the estimate can be controlled when summing over all $n$, if certain convexity properties hold. Most terms form telescope sums, such that only the first and the last summand remains. 

\begin{proof}[Proof of Theorem \ref{thm:convergence}]
  The $\omega$-semiconvexity of $F$ along with the strong convexity of $G$ allows us to estimate that
  \begin{align*}
    &\langle \partial G(u^{n+1})-\partial G(\hat{u}), u^{n+1} - \hat{u} \rangle-\frac{\sigma}{2}\|Ku_e^{n+1}\|^2   \\
    &+ \langle \partial F(g^{n+1}) - \partial F(\hat{g}), g^{n+1} - \hat{g}\rangle +\frac{\sigma}{2}\|g_e^{n+1}\|^2 \\
    \geq& c\|u^{n+1} - \hat{u}\|^2 -\frac{\sigma}{2}\|K(u^{n+1} - \hat{u})\|^2  - \omega \|g^{n+1} - \hat{g} \|^2 +\frac{\sigma}{2}\|g^{n+1} - \hat{g}\|^2 \\
    \geq&   \left(c - \omega \|K\|^2\right)\|u^{n+1} - \hat{u}\|^2.
  \end{align*}
  We start with the estimate derived in Lemma \ref{lem:estimate2},  keep the first row unchanged, use $\frac{\sigma \theta^2 \|K\|^2}{2} \leq \frac{1}{2\tau}$ in the second row, state that the third row is nonnegative due to  $\tau \sigma \|K\|^2 \leq 1$, and use the above estimate to state that 
  \begin{align*}
    0 \geq & \frac{1}{2\tau}(\|u_e^{n+1}\|^2 - \|u_e^n\|^2)  + \frac{1}{2\sigma} (\|q_e^{n+1}\|^2 - \|q_e^n\|^2) \\
    &+ \frac{1}{2\tau} \left(\| u_e^{n+1} - u_e^n \|^2 - \|u_e^n - u_e^{n-1}\|^2\right) \\
    &+ \theta \left( G(u^n) - G(u^{n-1}) -  \langle \hat{q}, K(u^{n-1}-u^n) \rangle \right) \\
    &+ \langle q_e^{n+1}, Ku_e^{n+1} \rangle - \langle q_e^{n}, Ku_e^{n} \rangle \\
    &+\frac{\sigma}{2}\left(\|Ku_e^{n+1}\|^2 -\|Ku_e^{n}\|^2 \right) - \frac{\sigma}{2} \theta \left(\|Ku_e^{n}\|^2 - \|Ku_e^{n-1}\|^2 \right)\\
    & + \left(c - \omega \|K\|^2\right)\|u^{n+1} - \hat{u}\|^2.
  \end{align*}
  Summing over the above inequality from $n=1$ to $n=N$ yields
  \begin{align*}
    0 \geq & \frac{1}{2\tau}(\|u_e^{N+1}\|^2 - \|u_e^1\|^2)  + \frac{1}{2\sigma} (\|q_e^{N+1}\|^2 - \|q_e^1\|^2) \\
    &+ \frac{1}{2\tau} \left(\| u_e^{N+1} - u_e^N \|^2 - \|u_e^1 - u_e^{0}\|^2\right) \\
    &+ \theta \left( G(u^N) - G(u^{0}) -  \langle \hat{q}, K(u^{0}-u^N) \rangle \right) \\
    &+ \langle q_e^{N+1}, Ku_e^{N+1} \rangle - \langle q_e^{1}, Ku_e^{1} \rangle \\
    &+\frac{\sigma}{2}\left(\|Ku_e^{N+1}\|^2 -\|Ku_e^{1}\|^2 \right) - \frac{\sigma}{2} \theta \left(\|Ku_e^{N}\|^2 - \|Ku_e^{0}\|^2 \right)\\
    & + \left(c - \omega \|K\|^2\right)\sum_{n=1}^N\|u^{n+1} - \hat{u}\|^2.
  \end{align*}
  Denoting the constant 
  \begin{align*}
    \kappa &=  \frac{1}{2\tau}\|u_e^1\|^2 + \frac{1}{2\sigma}  \|q_e^1\|^2 +  \frac{1}{2\tau} \|u_e^1 - u_e^{0}\|^2+ \theta G(u^{0}) \\
    &+ \langle \hat{q}, Ku^{0} \rangle +  \langle q_e^{1}, Ku_e^{1} \rangle +  \frac{\sigma}{2} \left(\|Ku_e^{1}\|^2 - \theta\|Ku_e^{0}\|^2 \right),
  \end{align*}
  we have
  \begin{align}
    \label{eq:temp}
    \kappa \geq & \frac{1}{2\tau}\|u_e^{N+1}\|^2 + \frac{1}{2\sigma} \|q_e^{N+1}\|^2 + \frac{1}{2\tau} \| u_e^{N+1} - u_e^N \|^2 \nonumber \\
    &+ \theta \left( G(u^N)+  \langle K^T\hat{q}, u^N \rangle \right) 
    + \langle q_e^{N+1}, Ku_e^{N+1} \rangle \nonumber \\
    &+\frac{\sigma}{2}\left(\|Ku_e^{N+1}\|^2 -\theta\|Ku_e^{N}\|^2 \right)  + \left(c - \omega \|K\|^2\right)\sum_{n=1}^N \|u^{n+1} - \hat{u}\|^2.
  \end{align}
  We observe that
  \begin{align}
    \label{eq:temp2}
    \langle q_e^{N+1}, Ku_e^{N+1} \rangle &\geq  -\|K\|\|q_e^{N+1}\|\|u_e^{N+1}\| \geq -\frac{1}{\sqrt{\tau \sigma}}\|q_e^{N+1}\|\|u_e^{N+1}\| \nonumber \\
    &\geq - \left(\frac{1}{2\sigma} \|q_e^{N+1}\|^2 + \frac{1}{2\tau}\|u_e^{N+1}\|^2 \right),
  \end{align}
  where we used the inequality $2ab \leq a^2 + b^2$ with $a = \frac{1}{\sqrt{\tau}}\norm{u_e^{N+1}}$ and $b = \frac{1}{\sqrt{\sigma}}\norm{q_e^{N+1}}$ in the last step.

  Additionally, any stationary point $(\hat{u}, \hat{g}, \hat{q})$ meets $-K^T\hat{q} \in \partial G(\hat{u})$. Since $G$ is strongly convex, we have
  \begin{align}
    \label{eq:temp3}
    G(u^N) - G(\hat{u}) &- \langle -K^T\hat{q}, u^N - \hat{u} \rangle \geq \frac{c}{2} \| u^N - \hat{u} \|^2, \nonumber \\
    \Rightarrow G(u^N) +  \langle K^T\hat{q}, u^N \rangle &\geq G(\hat{u}) - \langle -K^T\hat{q}, \hat{u} \rangle + \frac{c}{2} \| u^N - \hat{u} \|^2 \nonumber \\
    &\geq G(\hat{u}) - \langle -K^T\hat{q}, \hat{u} \rangle + \frac{\sigma}{4} \| Ku_e^N \|^2.
  \end{align}
  Using \eqref{eq:temp2} and \eqref{eq:temp3} in \eqref{eq:temp}, we obtain
  \begin{align}
    \kappa - \theta(G(\hat{u}) + \langle K^T\hat{q}, \hat{u} \rangle) \geq  \frac{\sigma}{2}\|Ku_e^{N+1}\|^2 -  \frac{\sigma \theta}{4} \|Ku_e^{N}\|^2 + \left(c - \omega \|K\|^2\right)\sum_{n=1}^N \|u^{n+1} - \hat{u}\|^2.
  \end{align}
  The left hand side is independent of $N$. If we choose $\theta=0$, the right hand side is nonnegative and since $(c - \omega \|K\|^2)>0$, the sum $\sum_{n=1}^N \|u^{n+1} - \hat{u}\|^2$ is bounded such that $\|u^{n+1} - \hat{u}\|^2 \rightarrow 0$ as least as fast as $1/N$. For $\theta >0$ we can first conclude the boundedness of $\|Ku_e^N\|^2$ with the help of 
  $$  \kappa - \theta(G(\hat{u}) + \langle K^T\hat{q}, \hat{u} \rangle) \geq 2\left( \frac{\sigma}{4}\|Ku_e^{N+1}\|^2\right) - \left( \frac{\sigma}{4}\|Ku_e^{N}\|^2\right),$$
and then make the same argument as above.
\end{proof}

\section{Experimental Convergence Results}
\label{lbl:appendixConv}
\figConvergences
In Fig.~\ref{fig:Convergences} we show experimental convergence to a critical point for the three examples where we don't have a full convergence theorem yet. 

\paragraph{Mumford-Shah Denoising.}
In the first row of Fig.~\ref{fig:Convergences} we show that $\norm{u^{n+1}-u^n}$ and $\norm{q^{n+1} - q^n}$ approach zero for the Mumford-Shah Denoising example. The input image and parameters were chosen as in Fig.~\ref{fig:dama} with varying $\lambda$. Using Proposition~\ref{prop:boundedness} and Proposition~\ref{lem:distances} we experimentally verify convergence to a critical point of our energy functional.

\paragraph{Mumford-Shah Inpainting and Image Dithering.}
Similarly to the previous paragraph, we also experimentally verify convergence for the Mumford-Shah Inpainting and Dithering examples in Fig.~\ref{fig:Convergences}. Parameters were chosen as described in according sections.

\end{appendix}

\end{document}